\RequirePackage[l2tabu, orthodox]{nag}
\documentclass[
final,11pt,reqno
]{amsart}

\usepackage{tikz}
\usepackage{microtype}
\makeatletter
\def\MT@register@subst@font{\MT@exp@one@n\MT@in@clist\font@name\MT@font@list
   \ifMT@inlist@\else\xdef\MT@font@list{\MT@font@list\font@name,}\fi}
\makeatother

\usepackage[notref,notcite]{showkeys}

\usepackage{mparhack,fixltx2e}

\usepackage{amsmath, amssymb, amsfonts, amsthm, mathtools}
\usepackage[all,warning]{onlyamsmath}
\usepackage{fixmath}
\mathtoolsset{centercolon}
\allowdisplaybreaks

\usepackage[latin1]{inputenc}
\usepackage[T1]{fontenc}
\usepackage{textcomp}

\usepackage[draft=false,pdftex]{hyperref}
\hypersetup{
    bookmarks=true,         
    unicode=false,          
    pdftoolbar=true,        
    pdfmenubar=true,        
    pdffitwindow=false,     
    pdfstartview={FitH},    
    pdftitle={Sets of unit vectors with small subset sums},    
    pdfauthor={Konrad J. Swanepoel},     
    pdfnewwindow=false,      
    colorlinks=true,       
    linkcolor=black,          
    citecolor=black,        
    filecolor=black,      
    urlcolor=black           
}

\theoremstyle{plain}
\newtheorem{theorem}{Theorem}
\newtheorem{lemma}[theorem]{Lemma}
\newtheorem{corollary}[theorem]{Corollary}
\newtheorem{conjecture}[theorem]{Conjecture}
\newtheorem{proposition}[theorem]{Proposition}
\newtheorem{definition}[theorem]{Definition}
\newtheorem{question}[theorem]{Question}

\newcommand{\set}[1]{\left\{#1\right\}}
\newcommand{\setbuilder}[2]{\left\{#1\,\middle|\,#2\right\}}
\newcommand{\BMdist}[2]{d_{\mathrm{BM}}(#1,#2)}
\newcommand{\lin}[1]{\operatorname{span}(#1)}
\newcommand{\rank}[1]{\operatorname{rank}(#1)}
\newcommand{\trace}[1]{\operatorname{tr}(#1)}
\newcommand{\transpose}[1]{#1^{\mathsf{T}}}
\newcommand{\epsi}{\varepsilon}
\newcommand{\norm}[1]{\left\lVert#1\right\rVert}

\newcommand{\Bignorm}[1]{\Bigl\lVert#1\Bigr\rVert}
\newcommand{\biggnorm}[1]{\biggl\lVert#1\biggr\rVert}
\newcommand{\ipr}[2]{\left\langle #1, #2 \right\rangle}
\newcommand{\Bigipr}[2]{\Bigl\langle #1, #2 \Bigr\rangle}
\newcommand{\abs}[1]{\left\lvert#1\right\rvert}
\newcommand{\Bigabs}[1]{\Bigl\lvert#1\Bigr\rvert}
\newcommand{\bigabs}[1]{\bigl\lvert#1\bigr\rvert}
\newcommand{\card}[1]{\left\lvert#1\right\rvert}
\newcommand{\conj}[1]{\overline{#1}}
\newcommand{\numbersystem}[1]{\mathbb{#1}}
\newcommand{\bF}{\numbersystem{F}}
\newcommand{\bN}{\numbersystem{N}}
\newcommand{\bR}{\numbersystem{R}}
\DeclareMathOperator{\conv}{conv}
\DeclareMathOperator*{\avg}{avg}

\newcommand{\rme}{\mathrm{e}}
\newcommand{\collection}[1]{{\mathcal#1}}
\newcommand{\CP}{\collection{P}}
\newcommand{\boundary}{\partial}
\newcommand{\vol}[1]{\operatorname{vol}\left(#1\right)}
\newcommand{\Bigvol}[1]{\operatorname{vol}\Bigl(#1\Bigr)}
\DeclareMathOperator{\diam}{diam}
\DeclareMathOperator{\interior}{int}
\newcommand{\vect}[1]{{\mathbold#1}}

\newcommand{\vc}{\vect{c}}

\newcommand{\ve}{\vect{e}}

\newcommand{\vo}{\vect{o}}
\newcommand{\vp}{\vect{p}}

\newcommand{\vs}{\vect{s}}

\newcommand{\vu}{\vect{u}}

\newcommand{\vx}{\vect{x}}
\newcommand{\vy}{\vect{y}}

\usepackage{xspace}
\newcommand{\dimensional}{\nobreakdash-\hspace{0pt}dimensional\xspace}
\newcommand{\C}{\mathcal{C}}
\newcommand{\CB}{\mathcal{CB}}
\newcommand{\Cupper}{\overline{\mathcal{C}}}
\newcommand{\Clower}{\underline{\mathcal{C}}}
\newcommand{\CBupper}{\overline{\mathcal{CB}}}
\newcommand{\CBlower}{\underline{\mathcal{CB}}}

\newcommand{\define}[1]{\emph{#1}}

\begin{document}
\title{Sets of unit vectors with small subset sums}
\author{Konrad J.\ Swanepoel}
\address{Department of Mathematics, London School of Economics and Political Science, Houghton Street, London WC2A 2AE, United Kingdom}
\email{\href{mailto:k.swanepoel@lse.ac.uk}{k.swanepoel@lse.ac.uk}}


\begin{abstract}
We say that a family $\setbuilder{\vx_i}{i\in[m]}$ of vectors in a Banach space $X$ satisfies the \emph{$k$-collapsing condition} if $\norm{\sum_{i\in I}\vx_i}\leq 1$ for all $k$-element subsets $I\subseteq\set{1,2,\dots,m}$.
Let $\Cupper(k,d)$ denote the maximum cardinality of a $k$-collapsing family of unit vectors in a $d$\dimensional Banach space, where the maximum is taken over all spaces of dimension $d$.
Similarly, let $\CBupper(k,d)$ denote the maximum cardinality if we require in addition that $\sum_{i=1}^m\vx_i=\vo$.
The case $k=2$ was considered by F\"uredi, Lagarias and Morgan (1991).
These conditions originate in a theorem of Lawlor and Morgan (1994) on geometric shortest networks in smooth finite\dimensional Banach spaces.
We show that $\CBupper(k,d)=\max\set{k+1,2d}$ for all $k,d\geq 2$.
The behaviour of $\Cupper(k,d)$ is not as simple, and we derive various upper and lower bounds for various ranges of $k$ and $d$.
These include the exact values $\Cupper(k,d)=\max\set{k+1,2d}$ in certain cases.

We use a variety of tools from graph theory, convexity and linear algebra in the proofs: in particular the Hajnal-Szemer\'edi Theorem, the Brunn-Minkowski inequality, and lower bounds for the rank of a perturbation of the identity matrix.
\end{abstract}

\maketitle
\setcounter{tocdepth}{1}
\tableofcontents

\setcounter{section}{-1}

\section{Notation}
Let $[n]$ denote the set $\set{1,2\dots,n}$, $\card{A}$ the cardinality of the set $A$, and $\binom{S}{k}$ the set $\setbuilder{A\subseteq S}{\card{A}=k}$ of $k$-subsets of $S$.
Let $d\geq 2$ and $m>k\geq 2$ be integers.
Given expressions $f(n)$ and $g(n)$ that depend (in particular) on $n\in\bN$, we use the notation $f(n)=O(g(n))$ or $f(n) \ll g(n)$ to show that $f\leq Cg$ for some absolute constant and sufficiently large $n$, and $f=o(g)$ or $f \lll g$ to indicate that $f/g\to 0$ as $n\to\infty$.

Let $X=X^d$ denote a $d$\dimensional real Banach space with norm $\norm{\cdot}$.
We denote the convex hull of a subset $A\subseteq X$ by $\conv(A)$, and the boundary of $A$ by $\boundary A$.
Throughout the paper we use the term \define{Minkowski space} for finite\dimensional real Banach space.
Denote the closed ball with centre $\vc$ and radius $r$ by \[B(\vc,r)=\setbuilder{\vx\in X}{\norm{\vx-\vc}\leq r}\text{.}\]
The \define{unit ball} of $X$ is $B_X:=B(\vo,1)$.
Denote the dual of $X$ by $X^*$.
The elements of $X^*$ are the \define{linear functionals} over $X$, that is, linear functions \[\vx^*\colon X\to\bR, \quad\vx\mapsto\ipr{\vx^*}{\vx}\text{,}\]
with norm
\[ \norm{\vx^*}^*:=\sup\setbuilder{\ipr{\vx^*}{\vx}}{\vx\in B_X}\text{.}\]
Any $\vx\in X\setminus\{\vo\}$ has a \define{dual unit vector}: a functional $\vx^*\in X^*$ such that $\norm{\vx^*}^*=1$ and $\ipr{\vx^*}{\vx}=\norm{\vx}$.
It is well-known that if the norm of a finite\dimensional $X$ is \define{smooth}, that is, if $\norm{\cdot}$ is differentiable on $X\setminus\set{\vo}$, then $X^*$ is strictly convex, that is, the boundary of $B_{X^*}$ does not contain a line segment.
Also, if $X$ is strictly convex, then $X^*$ is smooth.
Recall that a space is smooth iff any $\vx\in X\setminus\{\vo\}$ has a unique dual unit vector.

Denote the (multiplicative) \define{Banach-Mazur distance} between two Minkowski spaces $X$ and $Y$ of the same dimension by $\BMdist{X}{Y}$.

Denote the coordinates of $\vx\in\bR^d$ by $\vx=(\vx(1),\dots,\vx(d))$.
Let $p\in(1,\infty)$.
The space $\bR^d$ with the norm \[\norm{\vx}_p=\norm{(\vx(1),\vx(2),\dots,\vx(d))}_p:=\Big(\sum_{i=1}^d\abs{\vx(i)}^p\Bigr)^{1/p}\] is denoted by $\ell_p^d$ and the space $\bR^d$ with the norm \[\norm{\vx}_\infty=\norm{(\vx(1),\vx(2),\dots,\vx(d))}_\infty:=\max\setbuilder{\abs{\vx(i)}}{i\in[d]}\] by $\ell_\infty^d$.

\section{Introduction}
\begin{definition}\label{def1}
A family $\setbuilder{\vx_i}{i\in[m]}$ of $m$ \textup{(}not necessarily distinct\textup{)} vectors in some Minkowski space $X$ satisfies the \define{$k$-collapsing condition} if
\[ \quad\Bignorm{\sum_{i\in I}\vx_i}\leq 1\quad\text{for all $I\in\binom{[m]}{k}$,} \]
the \define{full collapsing condition}
\[ \quad\Bignorm{\sum_{i\in I}\vx_i}\leq 1\quad\text{for all $I\subseteq[m]$,} \]
the \define{strong balancing condition} if
\[ \sum_{i=1}^m\vx_i=\vo\text{,} \]
and the \define{weak balancing condition} if
\[ \vo\text{ is in the relative interior of }\conv\setbuilder{\vx_i}{i\in[m]}\text{.}\]
\end{definition}
In this paper we study the $k$-collapsing condition with or without the strong balancing condition.
In previous work by F\"uredi, Lagarias, Morgan, Lawlor and the present author \cite{FLM, LM, Sw96, Sw2} the full collapsing condition and the $2$-collapsing condition with or without the strong or the weak balancing condition were considered.
In Section~\ref{previouswork} we survey these previous results in order to sketch a context for the work presented in this paper.
New results are summarised in Section~\ref{newresults}.
Section~\ref{organisation} contains an overview of the remainder of this paper.

\subsection{Previous work}\label{previouswork}
The full collapsing and strong balancing conditions of Definition~\ref{def1} originate in a theorem of Lawlor and Morgan \cite{LM} on geometric shortest networks in smooth Minkowski spaces.
We next describe their work.

Given a family $N=\setbuilder{\vp_i}{i\in[n]}$ of points in a Minkowski space $X$, a \define{Steiner tree} is a (finite) tree $T=(V,E)$ such that $N\subseteq V\subset X$.
The points in $V\setminus N$ (if any) are called the \define{Steiner points} of $T$.
The \define{length} $\ell(T)$ of a tree is the sum $\sum_{\vx\vy\in E}\norm{\vx-\vy}$ of the edge lengths.
A \define{Steiner minimal tree} of $N$ is a Steiner tree of $N$ that minimises $\ell(T)$.
By a compactness argument \cite{Cockayne} any finite family of points in a Minkowski space has at least one Steiner minimal tree.
The following theorem characterises the edges that are incident to a Steiner point of a Steiner minimal tree when the underlying Minkowski space is smooth.
\begin{theorem}[Lawlor and Morgan \cite{LM}]\label{lmtheorem}
Let $N=\setbuilder{\vp_i}{i\in[n]}$ be a family of points, all different from the origin $\vo$, in a smooth Minkowski space $X$.
Let $\vp_i^*$ be the dual unit vector of $\vp_i$, $i\in[n]$.
Then the Steiner tree that joins $\vo$ to each $\vp_i$ by straight-line segments is a Steiner minimal tree of 
$N$ if and only if the family $\setbuilder{\vp_i^*}{i\in[n]}$ satisfies the full collapsing condition and the strong balancing condition in the dual space $X^*$.
\end{theorem}
Since the dual of a smooth Minkowski space is strictly convex, a natural problem suggested by Theorem~\ref{lmtheorem} is to find an upper bound on the cardinality of a family of unit vectors satisfying the full collapsing and strong balancing conditions in a strictly convex Minkowski space.
\begin{theorem}[Lawlor and Morgan \cite{LM}]\label{lmtheorem2}
Let $N=\setbuilder{\vx_i}{i\in[n]}$ be a family of unit vectors satisfying the full collapsing condition and the strong balancing condition in a $d$\dimensional strictly convex Minkowski space.
Then $n\leq d+1$.
\end{theorem}
Combined with Theorem~\ref{lmtheorem} this implies that the degree of a Steiner point in any Steiner minimal tree in a $d$\dimensional smooth Minkowski space is bounded above by $d+1$.

The following theorem characterises the edges incident to an arbitrary point of a Steiner minimal tree in a smooth Minkowski space.
Observe that if $\vp$ is a Steiner point of a Steiner minimal tree $T=(V,E)$ of the point family $N$, then $T$ is still a Steiner minimal tree of $N\cup\set{\vp}$ (but with $\vp$ not a Steiner point anymore).
Therefore, the condition in this characterisation should be logically weaker than the characterisation appearing in Theorem~\ref{lmtheorem}, and it turns out that the full balancing condition has to be dropped.
\begin{theorem}[\cite{Sw2}]\label{stheorem}
Let $N=\setbuilder{\vp_i}{i\in[n]}$ be a family of points, all different from the origin $\vo$, in a smooth Minkowski space $X$.
Let $\vp_i^*$ be the dual unit vector of $\vp_i$, $i\in[n]$.
Then the Steiner tree that joins $\vo$ to each $\vp_i$ by straight-line segments is a Steiner minimal tree of 
$N\cup\set{\vo}$ if and only if the family $\setbuilder{\vp_i^*}{i\in[n]}$ satisfies the full collapsing condition in the dual space $X^*$.
\end{theorem}
The following is a strengthening of Theorem~\ref{lmtheorem2}:
\begin{theorem}[\cite{Sw2}]\label{stheorem2}
Let $N=\setbuilder{\vx_i}{i\in[n]}$ be a family of unit vectors in a $d$\dimensional strictly convex Minkowski space satisfying the strong collapsing condition.
Then $n\leq d+1$.
\end{theorem}
Therefore, all points in Steiner minimal tree in a smooth $d$\dimensional Minkowski space have degree at most $d+1$.
Generalising Theorems~\ref{lmtheorem} and \ref{stheorem} to non-smooth Minkowski spaces is much more involved.
There the degrees of Steiner points can be as large as $2^d$; see \cite{Sw07} for a further discussion.
We now leave the original motivation of Steiner minimal trees behind and continue to survey previous work on the various collapsing and balancing conditions.

After the paper of Lawlor and Morgan \cite{LM}, F\"uredi, Lagarias and Morgan \cite{FLM} introduced the $2$-collapsing and weak balancing conditions, and used classical combinatorial convexity to study these conditions.
They showed the following.
\begin{theorem}[F\"uredi, Lagarias and Morgan \cite{FLM}]\label{flmtheorem3}
Let $N=\setbuilder{\vx_i}{i\in[n]}$ be a family of unit vectors in a $d$\dimensional Minkowski space $X$ satisfying the $2$-collapsing and weak balancing conditions.
Then $n\leq 2d$, with equality only if $N$ consists of a basis of $X$ and its negative.
\end{theorem}
They also mention without proof that if $N$ is a family of $2d$ unit vectors in a $d$\dimensional Minkowski space satisfying the full collapsing and the strong balancing condition, then the space is isometric to $\ell_\infty^d$.
We extend the above theorem to the $k$-collapsing condition, requiring however the strong balancing condition instead of the weak one (Theorem~\ref{balancedthm}).
The proof is completely different.

For strictly convex norms F\"uredi, Lagarias and Morgan \cite{FLM} obtained the following stronger conclusion (thus weakening the hypotheses of Theorem~\ref{lmtheorem2} in a different way from Theorem~\ref{stheorem2}).
\begin{theorem}[F\"uredi, Lagarias and Morgan \cite{FLM}]\label{flmtheorem4}
Let $N=\setbuilder{\vx_i}{i\in[n]}$ be a family of unit vectors in a $d$\dimensional strictly convex Minkowski space satisfying the $2$-collapsing condition and the weak balancing condition.
Then $n\leq d+1$.
\end{theorem}
Without any balancing condition or condition on the norm, they showed the following:
\begin{theorem}[F\"uredi, Lagarias and Morgan \cite{FLM}]\label{flmtheorem5}
Let $N=\setbuilder{\vx_i}{i\in[n]}$ be a family of unit vectors in a $d$\dimensional Minkowski space $X$ satisfying the $2$-collapsing condition.
Then $n\leq 3^d-1$.
\end{theorem}
This exponential behaviour for the $2$-collapsing condition without any balancing condition is necessary:
\begin{theorem}[F\"uredi, Lagarias and Morgan \cite{FLM}]\label{flmtheorem6}
For each sufficiently large $d\in\bN$ there exists a strictly convex and smooth $d$\dimensional Minkowski space with a family $N$ of at least $1.02^d$ unit vectors that satisfies the following strengthened $2$-collapsing condition: $\norm{\vx+\vy}<1$ for all $\set{\vx,\vy}\in\binom{N}{2}$.
\end{theorem}
We construct similar exponential lower bounds for the $k$-collapsing condition (Theorem~\ref{lbthm}).

In an earlier paper \cite{Sw96} we applied the Brunn-Minkowski inequality to improve the upper bound of Theorem~\ref{flmtheorem5} as follows.
\begin{theorem}[\cite{Sw96}]\label{stheorem4}
Let $N=\setbuilder{\vx_i}{i\in[n]}$ be a family of unit vectors in a $d$\dimensional Minkowski space $X$ satisfying the $2$-collapsing condition.
Then $n\leq 2^{d+1}+1$.
\end{theorem}
In this paper we combine the Brunn-Minkowski inequality with the Haj\-nal-Szemer\'edi Theorem from Graph Theory to extend the above theorem to the $k$-collapsing condition (Theorem~\ref{bmthm}).
In \cite{FLM} it was asked whether there is an upper bound polynomial in $d$ for the size of a collection of unit vectors in a $d$\dimensional Minkowski space satisfying the strong collapsing condition but not necessarily any balancing condition.
This was subsequently answered as follows:
\begin{theorem}[\cite{Sw96}]\label{Sw96theorem}
Let $N=\setbuilder{\vx_i}{i\in[n]}$ be a family of unit vectors in a $d$\dimensional Minkowski space $X$ satisfying the strong collapsing condition.
Then $n\leq 2d$, with equality if and only if $X$ is isometric to $\ell_\infty^d$,  with $N$ corresponding to $\setbuilder{\pm\ve_i}{i\in[d]}$ under any isometry.
\end{theorem}
The analogous theorem for the strictly convex case is as follows:
\begin{theorem}[\cite{Sw2}]\label{stheorem3}
Let $N=\setbuilder{\vx_i}{i\in[n]}$ be a family of unit vectors in a $d$\dimensional strictly convex Minkowski space $X$ satisfying the full collapsing condition.
Then $n\leq d+1$.
If, in addition, the balancing condition is not satisfied then $n\leq d$.
\end{theorem}
The full collapsing condition is closely connected to certain notions from the local theory of Banach spaces.
The \define{absolutely summing constant} or the \define{$1$-summing constant} $\pi_1(X)$ of a Minkowski space $X$ is defined to be the infimum of all $c>0$ satisfying
\[\sum_{i=1}^m\norm{\vx_i}\leq c\max_{\epsilon_i=\pm 1}\Bignorm{\sum_{i=1}^m\epsilon_i\vx_i}\]
where $\vx_1,\dots,\vx_m\in X$.
It is clear that $2\pi_1(X)$ is an upper bound to the number of unit vectors that satisfy the full collapsing condition.
Deschaseaux \cite{Deschaseaux} showed that $\pi_1(X)\leq d$ with equality iff $X$ is isometric to $\ell_\infty^d$.
This gives another proof of Theorem~\ref{Sw96theorem}, apart from the characterisation of the family of unit vectors in the case of equality.
Franchetti and Votruba \cite{FV} showed that if $X$ is $2$\dimensional then $2\pi_1(X)$ equals the perimeter of the unit circle.
By a result of Go\l{}\k{a}b \cite{MSW}, the perimeter of the unit circle is less than $4$ unless $X$ is isometric to $\ell_\infty^2$.
This implies the $2$\dimensional case of Deschaseaux's theorem.

For $q\geq 2$, the \define{cotype $q$ constant $\kappa_q(X)$} of a Minkowski space $X$ is defined to be the infimum of all $c>0$ such that
\[\Bigl(\sum_{i=1}^m\norm{\vx_i}^q\Bigr)^{1/q}\leq c\avg_{\epsilon_i=\pm 1}\Bigl(\Bignorm{\sum_{i=1}^m\epsilon_i\vx_i}^2\Bigr)^{1/2}\]
where $\vx_1,\dots,\vx_m\in X$.
It is again straightforward that $(2\kappa_q(X))^q$ is an upper bound for the number of vectors satisfying the full collapsing condition.
For instance, bounds on the cotype $2$ constants for $\ell_p^d$ (essentially consequences of the Khinchin inequalities) give upper bounds independent of the dimension for fixed $p\in[1,\infty)$.
Details may be found in \cite{Sw2}.


A more general question was asked by Sidorenko and Stechkin \cite{SS1, SS2} and Katona and others \cite{Katona1, Katona2, Katona3, Katona4, KMW}, where the `$\leq 1$' in the collapsing conditions is replaced by `$\leq\delta$' or `$<\delta$'.
In this direction work was also done in \cite{Sw3}.
We do not pursue this generalisation here, instead leaving it for a later investigation, as it will be seen that the arguments in this paper are already quite involved.

\subsection{Overview of new results}\label{newresults}
In this paper we only consider the $k$-collapsing condition and strong balancing condition.
\begin{definition}\label{def13}
For any $k\geq 2$, define \define{$\C_k(X)$} to be the largest $m$ such that a family of $m$ vectors in $X$ of norm at least $1$ exists that satisfies the $k$-collapsing condition.
Also, define \define{$\CB_k(X)$} to be the largest $m$ such that a family of $m$ vectors in $X$ of norm at least $1$ exists that satisfies the $k$-collapsing condition and the strong balancing condition.

Next define the numbers
\begin{align*}
\Cupper(k,d)&:=\max \setbuilder{\C_k(X^d)}{\text{$X^d$ is a $d$-dimensional Minkowski space}}\text{,}\\
\Clower(k,d)&:=\min \setbuilder{\C_k(X^d)}{\text{$X^d$ is a $d$-dimensional Minkowski space}}\text{,}\\
\CBupper(k,d)&:=\max \setbuilder{\CB_k(X^d)}{\text{$X^d$ is a $d$-dimensional Minkowski space}}\text{,}\\
\CBlower(k,d)&:=\min \setbuilder{\CB_k(X^d)}{\text{$X^d$ is a $d$-dimensional Minkowski space}}\text{.}
\end{align*}
\end{definition}
A simple compactness argument shows that $\Cupper(k,d)$ and $\CBupper(k,d)$ are always finite.
Although the vectors occurring in Theorems~\ref{lmtheorem} to \ref{stheorem3} above are unit vectors, we weaken this to vectors of norm at least $1$ in the above definition.
Indeed it turns out that the quantities $\Cupper(k,d)$ and $\CBupper(k,d)$ stay exactly the same whether we require the vectors to be of norm $\geq 1$ or $=1$.
See Corollary~\ref{cor38} in Section~\ref{section:matrix} for this non-trivial fact.

Since we have assumed $d\geq 2$, it follows that for any value of $k\geq 2$ there exist $k+1$ unit vectors that satisfy the strong balancing condition, hence also the $k$-collapsing condition.
\begin{proposition}\label{prop:k+1}
Let $k,d\geq 2$.
Then $\C_k(X^d)\geq \CB_k(X^d)\geq k+1$ for any $d$\dimensional $X^d$.
\end{proposition}
In Section~\ref{section2} we show that these inequalities cannot be improved in general:
\begin{proposition}\label{cor:eucl}
$\C_k(\ell_2^d)=\CB_k(\ell_2^d)=k+1$ for any $k\geq 2$ and $d\geq 2$.
\end{proposition}
Consequently,
\begin{corollary}
$\Clower(k,d)=\CBlower(k,d)=k+1$ for all $k,d\geq 2$.
\end{corollary}
The family of $d$ unit vectors and their negatives $\set{\pm \ve_1,\dots,\pm \ve_d}$ shows the following:
\begin{proposition}\label{prop:inflb}
Let $k,d\geq 2$.
Then
\[\C_k(\ell_\infty^d)\geq \CB_k(\ell_\infty^d)\geq 2d.\]
\end{proposition}
\begin{corollary}
$\Cupper(k,d)\geq\CBupper(k,d)\geq\max\set{k+1,2d}$ for all $k,d\geq 2$.
\end{corollary}
In Section~\ref{section2} we show the following:
\begin{proposition}\label{inf}
For any $k\geq 2$ and $d\geq 2$, \[\C_k(\ell_\infty^d)=\CB_k(\ell_\infty^d)=\max\set{k+1,2d}\text{.}\]
\end{proposition}
It turns out that this is an extremal case for the quantity $\CB_k(X^d)$.
\begin{theorem}\label{balancedthm}
For any $k\geq 2$ and $d\geq 2$,
\[ \CBupper(k,d)=\max\set{k+1,2d}\text{.}\]

If $d\geq 2$, $2\leq k\leq 2d-2$ and $\CB_k(X^d)=2d$, then any family of $2d$ vectors of norm at least $1$ satisfying the $k$-collapsing and strong balancing conditions are necessarily unit vectors consisting of a basis of $X^d$ and its negative.
If furthermore $d\leq k\leq 2d-2$, then the only space $X^d$ for which $\CB_k(X^d)=2d$ is $\ell_\infty^d$ up to isometry.

\textup{(}If $2\leq k\leq d-1$ then there are infinitely many non-isometric spaces $X^d$ such that $\CB_k(X^d)=2d$.\textup{)}
\end{theorem}
%
Cf.\ Theorem~\ref{flmtheorem3} above.
The proof uses a reduction to $m\times m$ matrices that are perturbations of the identity matrix in a certain weak sense, together with results on lower bounds of the ranks of such matrices (Lemma~\ref{ranklemma}).
In order to apply these lower bounds we also have to solve certain convex optimization problems (Lemmas~\ref{optimisation1} and \ref{optimisation2}).
Analogous to Theorem~\ref{flmtheorem4} above we make the following conjecture.
\begin{conjecture}
If $X^d$ is a strictly convex $d$\dimensional Minkowski space then \[\CBupper_k(X^d)\leq\max\set{k+1,d+1}\text{.}\] 
\end{conjecture}
This conjecture holds for $k=2$ \cite{FLM}.
Also, for each $d\geq 2$ there exists a strictly convex $d$\dimensional space with $d+1$ unit vectors satisfying the strong collapsing condition, so this conjecture would give the best possible estimate if true.
Analogous to Theorem~\ref{flmtheorem3} we may hope for a positive answer to the following question.
\begin{question}
Can the strong balancing condition in Theorem~\ref{balancedthm} be replaced by the weak balancing condition?
That is, if the family $\set{\vx_1,\dots,\vx_m}$ of unit vectors in a $d$\dimensional Minkowski space $X^d$ satisfies the $k$-collapsing condition and weak balancing condition, is $m\leq\max\set{k+1,2d}$?
\end{question}
Our methods do not seem to offer any way of using the weak balancing condition.
Again, it is known that the answer is yes when $k=2$ \cite{FLM}.

Estimating $\Cupper(k,d)$ is much harder.
The same proof techniques work only up to a certain extent and the details become much trickier.

\begin{theorem}\label{rankthm1}\mbox{}
For $k\geq 2$ let $\gamma_k$ be the unique \textup{(}positive\textup{)} solution to
\[ (1+x)^{1/x}\left(1+\frac{1}{x}\right)=k^2\text{.}\]
Then $\rme/k^2<\gamma_k<\rme/(k^2-\rme)$ and
\begin{equation}\label{bound20}
\Cupper(k,d)< 1.33k^{2\gamma_k d+2}\text{.}
\end{equation}
If $k<\sqrt{d}$ then
\[\Cupper(k,d)< \frac{k}{\sqrt{d}}k^{2\gamma_k d+2}\text{.}\]
In particular, if $k = c\sqrt{d}$ with $c<1$, then $\Cupper(k,d)=O(d^{1+\rme/c^2})$ as $d\to\infty$.
\end{theorem}
See Table~\ref{table1} for the first few values of $\gamma_k$.
\begin{table}
\centering
\begin{tabular}{lllll}
$k$ & $\gamma_k$ & $k^{2\gamma_k d}$ & $(1+\frac{2}{k})^d$ & $\left(1+\frac{1}{2(2k+1)^2}\right)^d$ \\[1mm]\hline \\[-3mm]
$2$ & $1$         & $4^d$     & \boldmath$2^d$\unboldmath     & $1.02^d$\\
$3$ & $0.3541686$ & $2.178^d$ & \boldmath$1.667^d$\unboldmath & $1.0102^d$\\
$4$ & $0.1854203$ & $1.673^d$ & \boldmath$1.5^d$\unboldmath   & $1.0061^d$\\
$5$ & $0.1149225$ & $1.448^d$ & \boldmath$1.4^d$\unboldmath   & $1.0041^d$\\
$6$ & $0.0784510$ & \boldmath$1.325^d$\unboldmath & $1.334^d$ & $1.0029^d$\\
$7$ & $0.0570503$ & \boldmath$1.249^d$\unboldmath & $1.286^d$ & $1.0022^d$\\
$8$ & $0.0433914$ & \boldmath$1.198^d$\unboldmath & $1.25^d$  & $1.0017^d$\\
$9$ & $0.0341301$ & \boldmath$1.162^d$\unboldmath & $1.223^d$ & $1.0013^d$
\end{tabular}
\bigskip
\caption{Values of $\gamma_k$ with the upper bounds of Theorems~\ref{rankthm1} and \ref{bmthm} and the lower bound of Theorem~\ref{lbthm}. The values of $\gamma_k$ are rounded to the nearest decimal, of $k^{2\gamma_k}$ and $1+2/k$ are rounded up and of $1+1/(2(2k+1)^2)$ are rounded down.}\label{table1}
\end{table}
The next theorem gives a slightly sharper result for $k$ a small multiple of $\sqrt{d}$.
See also the lower bound of Theorem~\ref{lbthm2} below.
\begin{theorem}\label{thm:sqrtd}
For any $\epsi>0$ and $p\in\bN$, $p\geq 2$, there exist $d_0$ and $c>0$ such that for all $d>d_0$, if \[ \left((p!)^{-1/(2p)}+\epsi\right)\sqrt{d}<k\leq\sqrt{d}\] then $\Cupper(k,d)<cd^p$.
\end{theorem}
For larger $k$ we obtain almost optimal results.
In particular, we obtain the exact result $\Cupper(k,d)=2d$ for $(\sqrt{6}-2)d+O(1)<k<2d-\sqrt{d/2}$.
\begin{theorem}\label{rankthm2}\mbox{}
Let $k\geq 3$ and $d\geq 2$.
\begin{enumerate}
\item If $\sqrt{d} < k \leq \frac{d+1}{2}$ then $\Cupper(k,d)\leq \frac{2d(k-1)^2}{k^2-d}=2d\left(1+\frac{d-2k+1}{k^2-d}\right)$.
\item If $-2d+\sqrt{6d^2+3d+1} \leq k \leq 2d-\sqrt{d/2}$ then $\Cupper(k,d) = 2d$.
\item If $d\geq 3$ and $k>2d-\sqrt{d/2}$ then $\Cupper(k,d)\leq k+\frac{1+\sqrt{2d-3}}{2}$.
\end{enumerate}
\end{theorem}
For values of $d$ up to $7$ as $k\to\infty$ the same methods as used in proving Theorems~\ref{balancedthm}, \ref{rankthm1}, \ref{thm:sqrtd} and \ref{rankthm2} give the following exact values.
\begin{theorem}\label{newthm}
$\Cupper(k,d)=\max\set{k+1,2d}$ in the following cases:
\begin{enumerate}
\item $d=2$ and $k\geq 2$,
\item $d\in\set{3,4,5}$ and $k\geq 3$, 
\item $d=6$ and $k\in\set{3,4,5,\dots,10}\cup\set{17,18,19\dots}$,
\item $d=7$ and $k\in\set{3,4,5,\dots,12}\cup\set{41,42,43,\dots}$.
\end{enumerate}
\end{theorem}
The proof method gives no information for $d\geq 8$ and $k$ large.
(The estimate $\Cupper(2,3)\leq 9$ is also obtained in the proof.)
For arbitrary $d$, as long as $k$ is large, we obtain the following using a completely different technique.
\begin{theorem}\label{thm:klarge}
If $k\gg d^{d+2}$ then $\Cupper(k,d)=k+1$.
\end{theorem}
The proof uses geometric tools from convexity, in particular the Brunn-Minkowski inequality and the theorem of Carath\'eodory.
The hypothesis $k\gg d^{d+2}$ is most likely not best possible, but we need at least $k\geq 2d-1$ for the conclusion of this theorem to hold, as shown by the example of $k\leq 2d-2$ and the family $\setbuilder{\pm\ve_i}{i\in[d]}$ in $\ell_\infty^d$.
\begin{conjecture}
$\Cupper(k,d)=k+1$ whenever $k\geq 2d-1$.
\end{conjecture}
By Theorem~\ref{newthm} this conjecture holds for $d\leq 5$.
The next conjecture has non-empty content only for $d\geq 8$.
\begin{conjecture}\label{conj2}
$\Cupper(k,d)=2d$ if $2d-\sqrt{d/2}\leq k\leq 2d-2$.
\end{conjecture}
Since Theorem~\ref{rankthm2} gives $\Cupper(k,d)=2d$ for $(\sqrt{6}-2)d+O(1)<k<2d-\sqrt{d/2}$, it is likely that the  bound in Conjecture~\ref{conj2} already holds for values of $k$ smaller than $(\sqrt{6}-2)d$.
On the other hand, as implied by Theorem~\ref{lbthm2} below, we need at least $k>(\frac{1}{2}+o(1))\sqrt{d}$.

We show the following upper bound using a method closely related to the proof of Theorem~\ref{thm:klarge}.
We agin use the Brunn-Minkowski inequality, but combine it with the Hajnal-Szemer\'edi theorem from graph theory:
\begin{theorem}\label{bmthm}
For any $k,d\geq2$, $\Cupper(k,d)\leq k(1+\frac{2}{k})^d+k-1$.
\end{theorem}
Asymptotically for fixed $k$ as $d\to\infty$, this bound is better when $k\leq5$ while for $k\geq 6$ Theorem~\ref{rankthm1} is better.
See Table~\ref{table1} for a comparison between the upper bounds given by Theorem~\ref{rankthm1} and Theorem~\ref{bmthm} for $k=2,\dots,8$.

Related to Proposition~\ref{cor:eucl} is the following result on spaces close to Euclidean space.
\begin{proposition}\label{BMDistance}
Let $D=\BMdist{X^d}{\ell_2^d}$ be the Banach-Mazur distance between $X^d$ and $\ell_2^d$.
Then for any $k>D^2$, \[\C_k(X^d)\leq \frac{k^2-D^2}{k-D^2}=k+D^2+\frac{D^4-D^2}{k-D^2}\text{.}\]
In particular, if $D^2\leq(2k-1)/(k+1)$ then $\C_k(X^d)=k+1$.
\end{proposition}
Its simple proof is at the end of Section~\ref{section2}.
By John's theorem (see \cite{GS} for a modern account), 
$\BMdist{X^d}{\ell_2^d}\leq\sqrt{d}$,
from which follows $\C_k(X^d)\leq k+d+\frac{d^2-d}{k-d}$ if $k>d$.
This estimate is worse, however, than the estimates of Theorems~\ref{rankthm2} and \ref{newthm} whenever $k>d$.
On the other hand, if $D=\BMdist{X}{\ell_2^d}$ is sufficiently small, then Proposition~\ref{BMDistance} may give bounds better than Theorems~\ref{rankthm2}.
In particular, Proposition~\ref{BMDistance} is better than Theorem~\ref{rankthm2} in the range $d<k\leq 2d-\sqrt{d/2}$ if $\BMdist{X}{\ell_2^d}\leq\sqrt{\frac{(2d-k)k}{2d-1}}$, and in the range $k>2d-\sqrt{d/2}$ if $\BMdist{X}{\ell_2^d}\leq(d/2)^{1/4}$.

\bigskip
We now turn to lower bounds.
The first, generalising Theorem~\ref{flmtheorem6}, uses a simple greedy construction of sets of almost orthogonal Euclidean unit vectors.
\begin{theorem}\label{lbthm}
For all $k\geq 2$ and sufficiently large $d$ depending on $k$, there exists a strictly convex and smooth $d$\dimensional Minkowski space $X^d$ such that
\begin{equation}\label{bound29}
\C_k(X^d)\geq \left(1+\frac{1}{2(2k+1)^2}\right)^d\text{.}
\end{equation}
\end{theorem}
The proof in fact gives a norm that is $C^\infty$ on $\bR^d\setminus\set{\vo}$.
The lower bound \eqref{bound29} almost matches the upper bound \eqref{bound20} from Theorem~\ref{rankthm1} asymptotically in the sense that as $k\to\infty$ and $d \ggg \log k$, \eqref{bound29} implies that $\Cupper(k,d)^{1/d}-1 \gg 1/k^2$, while \eqref{bound20} implies that $\Cupper(k,d)^{1/d}-1 \ll (\log k)/k^2$.
See the last column in Table~\ref{table1}.
(Note that since $\Cupper(k,d)\geq k+1$, we need $d$ to grow with $k$ in order to have $\lim_{k\to\infty}\Cupper(k,d)^{1/d}=1$, and in fact $\lim_{k\to\infty}(k+1)^{1/d}=1$ iff $d\ggg\log k$.)

The second lower bound uses an algebraic construction of almost orthogonal Euclidean vectors.
\begin{theorem}\label{lbthm2}
For any $d\in\bN$ let $q=q_d$ be the largest prime power such that $d\geq q^2-q+1$.
\textup{(}By the Prime Number Theorem, $q_d\sim\sqrt{d}$ as $d\to\infty$.\textup{)}
Then for each $c\in\bN$ and $k\geq 2$ satisfying $c\leq q-2$ and
\[k\leq\frac{q-1}{2c}-\frac{1}{2}\quad \biggl(\sim\frac{\sqrt{d}}{2c}\biggr)\]
there exists a $d$\dimensional Minkowski space $X^d$ such that
\[\C_k(X^d)\geq q^{c+2}\quad(\sim d^{1+c/2}\text{ as $d\to\infty$})\text{.}\]
\end{theorem}
In particular, when $k\leq(\frac{1}{2}+o(1))\sqrt{d}$ as $d\to\infty$ we have $\Cupper_k(d)\gg d^{3/2}$.
The lower bound of Theorem~\ref{lbthm2} is better than that of Theorem~\ref{lbthm} when $k\gg\sqrt{d}/\log d$.
For $k$ a small multiple of $\sqrt{d}$, Theorems~\ref{rankthm1} and \ref{thm:sqrtd} give an upper bound polynomial in $d$ while Theorem~\ref{lbthm2} gives a lower bound polynomial in $d$, but with a gap between the degrees of the polynomials.
Nevertheless, Theorem~\ref{lbthm2} matches the bound \eqref{bound20} of Theorem~\ref{rankthm1} in a similar sense as in the discussion after Theorem~\ref{lbthm}, in that it implies that $\Cupper(k,d)^{1/d}-1\gg(\log k)/k^2$ as $k\to\infty$ and $k\sim\sqrt{d}/(2c)$, $c\in\bN$.

\subsection{Organisation of the paper}\label{organisation}
In Section~\ref{section2} we use elementary combinatorial arguments involving coordinates and inner products to prove Proposition~\ref{inf} on $\ell_\infty^d$, Proposition~\ref{cor:eucl} on $\ell_2^d$ and Proposition~\ref{BMDistance} on spaces close to $\ell_2^d$.
In Section~\ref{sectionBM} we use the Brunn-Minkowski inequality and the Hajnal-Szemer\'edi Theorem to prove Theorem~\ref{bmthm}.
This is followed in Section~\ref{sectionCara} by a proof of Theorem~\ref{thm:klarge} which is along similar lines.
In addition to the Brunn-Minkowski inequality it uses a metric consequence of Carath\'eodory's Theorem that may be of independent interest (Lemma~\ref{diameter}).
Then in Section~\ref{section:matrix} we reformulate the notion of a $k$-collapsing collection of vectors in terms of matrices.
There we also prove a general version of a well-known result that bounds the rank of a matrix from below (Lemma~\ref{ranklemma}).
These results are applied in Section~\ref{section:tight}, where Theorem~\ref{balancedthm} is proved, and Section~\ref{section:tight2} where Theorems~\ref{rankthm2} and \ref{newthm} are proved.
These proofs are all very technical and involve an application of Lemma~\ref{ranklemma} combined with convex optimisation.
In Section~\ref{section:power} Theorems~\ref{rankthm1} and \ref{thm:sqrtd} are proved.
The arguments are similar as in Sections~\ref{section:tight} and \ref{section:tight2} and use in addition a well-known bound on the rank of an integer Hadamard power of a matrix (Lemma~\ref{lemma:hpower}).
In Section~\ref{section:lowerbounds} we derive the lower bounds of Theorems~\ref{lbthm} and \ref{lbthm2}.

\section{The sup-norm and Euclidean norm}\label{section2}
\begin{proposition}\label{infprop}
Let $k,d\geq 2$.
If $S=\setbuilder{\vx_i}{i\in[m]}\subset\ell_\infty^d$ is a $k$-collapsing family of $m > k+1$ vectors of norm at least $1$, then $m\leq 2d$.
If furthermore $m=2d$, then $S=\set{\pm\ve_1,\dots,\pm\ve_d}$.
\end{proposition}
\begin{proof}
Suppose that there exist a coordinate $j\in[d]$ and two distinct indices $i\in[m]$ such that $\vx_i(j)\geq 1$.
Without loss of generality, $\vx_{m-1}(1), \vx_m(1)\geq 1$.
By the $k$-collapsing condition, for any $I\in\binom{[m-2]}{k-2}$,
\[\sum_{i\in I}\vx_i(1)\leq-2+\sum_{i\in I\cup\set{m-1,m}}\vx_i(1)\leq-2+\biggnorm{\sum_{i\in I\cup\set{m-1,m}}\vx_i}_\infty\leq -1\text{.}\]
Fix a $J\in\binom{[m-2]}{k}$ (note that $k\leq m-2$).
It follows that
\[ \binom{k-1}{k-3}\sum_{i\in J}\vx_i(1)=\sum_{I\in\binom{J}{k-2}}\sum_{i\in I}\vx_i(1)\leq -\binom{k}{k-2},\]
which gives
\[\sum_{i\in J} \vx_i(1)\leq -\binom{k}{k-2}/\binom{k-1}{k-3}=-k/(k-2) < -1,\]
hence $\norm{\sum_{i\in J}\vx_i}_\infty>1$, contradicting the $k$-collapsing condition.

Therefore, for each coordinate $j\in[d]$ there is at most one index $i\in[m]$ such that $\vx_i(j)\geq 1$.
Similarly, there is at most one $i\in[m]$ such that $\vx_i(j)\leq -1$.
Therefore, there are at most $2d$ pairs $(i,j)\in[m]\times[d]$ such that $\abs{\vx_i(j)}\geq 1$.
On the other hand, since $\norm{\vx_i}_\infty\geq 1$ for each $i\in[m]$, there are at least $m$ such pairs, which gives $m\leq 2d$.

If we assume $m=2d$, then for each $j\in[d]$ there is exactly one $i\in[m]$ such that $\vx_i(j)\geq 1$, and exactly one $i\in[m]$ such that $\vx_i(j)\leq -1$.
We may then renumber the $\vx_i$ such that $\vx_{2i-1}(i)\geq 1$ and $\vx_{2i}(i)\leq -1$ for each $i\in[d]$.
By the $k$-collapsing condition, for any $J\in\binom{[m-2]}{k-1}$,
\[ \sum_{i\in J}\vx_i(d) + 1\leq\sum_{i\in J\cup\set{2d-1}}\vx_i(d)\leq\biggnorm{\sum_{i\in J\cup\set{2d-1}}\vx_i}_\infty\leq 1\text{,}\]
hence $\sum_{i\in J}\vx_i(d)\leq 0$.
Similarly, $\sum_{i\in J}\vx_i(d)\geq 0$.
Therefore, $\sum_{i\in J}\vx_i(d)=0$ for each $J\in\binom{[m-2]}{k-1}$.
Since $k-1<m-2$, it follows that $\vx_i(d)=0$ for all $i\in[m-2]$ and $\vx_{2d-1}(d)=1$, $\vx_{2d}(d)=-1$.
Similarly, $\vx_i(j)=0$ for all $i,j$ such that $i\notin\set{2j-1, 2j}$, and $\vx_{2j-1}(j)=1$, $\vx_{2j}(j)=-1$.
We conclude that $\vx_{2i-1}=\ve_i$ and $\vx_{2i}=-\ve_i$ for all $i\in[d]$.
\end{proof}

\begin{proof}[Proof of Proposition~\ref{inf}]
By Propositions~\ref{prop:k+1} and \ref{prop:inflb},
$\C_k(\ell_\infty^d)\geq \CB_k(\ell_\infty^d)\geq\max\set{k+1,2d}$.
Proposition~\ref{infprop} implies that $\C_k(\ell_\infty^d)\leq\max\set{k+1,2d}$.
\end{proof}

The next lemma occurs in an equivalent form in \cite[Lemma~5]{Katona1}.
\begin{lemma}\label{eucl}
Let $k\geq 2$ and $\lambda\in(0,\sqrt{k})$.
Let $\vx_1,\dots,\vx_m$ be vectors in an inner product space such that $\norm{x_i}_2\geq 1$ for all $i\in[m]$ and
\begin{equation}\label{lambda}
\Bignorm{\sum_{i\in I} \vx_i}_2\leq\lambda\quad\text{for all}\quad I\in\binom{[m]}{k}.
\end{equation}
Then \[m\leq\frac{k^2-\lambda^2}{k-\lambda^2}.\]
\end{lemma}
\begin{proof}
Square \eqref{lambda} and sum over all $I\in\binom{[m]}{k}$ to obtain
\begin{align*}
\binom{m}{k}\lambda^2 & \geq \binom{m-1}{k-1}\sum_{i=1}^m\norm{\vx_i}_2^2 + \binom{m-2}{k-2}\sum_{\set{i,j}\in\binom{[m]}{2}}^m2\ipr{\vx_i}{\vx_j} \\
& = \left(\binom{m-1}{k-1}-\binom{m-2}{k-2}\right)\sum_{i=1}^m\norm{\vx_i}_2^2 + \binom{m-2}{k-2}\Bignorm{\sum_{i=1}^m \vx_i}_2^2\\
& \geq \left(\binom{m-1}{k-1}-\binom{m-2}{k-2}\right)m+0,
\end{align*}
which simplifies to the conclusion of the theorem.
\end{proof}

\begin{proof}[Proof of Proposition~\ref{cor:eucl}]
For the upper bound, set $\lambda=1$ in Lemma~\ref{eucl}.
The lower bound follows from Proposition~\ref{prop:k+1}.
\end{proof}

\begin{proof}[Proof of Proposition~\ref{BMDistance}]
By the definition of Banach-Mazur distance there exist coordinates such that
$\norm{\vx}\leq\norm{\vx}_2\leq D\norm{\vx}$ for all $\vx\in X^d$.
Then apply Lemma~\ref{eucl} with $\lambda=D$.
\end{proof}

\section{The Brunn-Minkowski inequality and graph colourings}\label{sectionBM}
The proofs of Theorems~\ref{thm:klarge} and \ref{bmthm} are similar, but that of Theorem~\ref{bmthm} is somewhat more straightforward and we consider it first.
We first discuss the three main tools used in its proof.
The first is the dimension-independent version of the Brunn-Minkowski inequality (see Ball \cite{Ball}.)
Denote the volume (or $d$\dimensional Lebesgue measure) of a measurable set $A\subseteq\bR^d$ by $\vol{A}$.
\theoremstyle{plain}
\newtheorem*{bmineq}{Brunn-Minkowski inequality}
\begin{bmineq}
If $A,B\subset\bR^d$ are compact sets and $0<\lambda<1$, then
\[\vol{\lambda A+(1-\lambda)B}\geq\vol{A}^{\lambda}\vol{B}^{1-\lambda}.\]
\end{bmineq}
Induction immediately gives the following version for $k$ sets:
\newtheorem*{bmineqk}{$k$-fold Brunn-Minkowski inequality}
\begin{bmineqk}
Let $A_1,A_2,\dots,A_k\subset\bR^d$ be compact and $\lambda_1,\lambda_2,\dots,\lambda_k>0$ such that $\sum_{i=1}^k\lambda_i=1$. Then
\[\vol{\lambda_1 A_1+\lambda_2 A_2+\dots+\lambda_k A_k}\geq\prod_{i=1}^k\vol{A_i}^{\lambda_i}\text{.}\]
\end{bmineqk}
The second tool is the Hajnal-Szemer\'edi Theorem.
A \define{$k$-colouring} of a graph $G=(V,E)$ is a function $f\colon V \to[k]$ such that $f(x)\neq f(y)$ whenever $xy\in E$.
The $k$-colouring partitions the vertex set $V$ into \define{colour classes} $f^{-1}(i)$, $i\in[k]$. 
A $k$-colouring of a graph on $m$ vertices is called \define{equitable} if each colour class has cardinality $\lfloor m/k\rfloor$ or $\lceil m/k\rceil$.
The following result was originally a conjecture of Erd\H{o}s \cite{Erdos}.
Although the original proof \cite{HS} was quite complicated and long, there is now a relatively simple, compact proof, due to Kierstead and Kostochka \cite{KK}.
\theoremstyle{plain}
\newtheorem*{hstheorem}{Hajnal-Szemer\'edi theorem}
\begin{hstheorem}
Let $G$ be a graph with maximum degree $\Delta$.
Then for any $k>\Delta$, $G$ has an equitable $k$-colouring.
\end{hstheorem}

The third tool is the following simple consequence of the triangle inequality.
\begin{lemma}\label{triangle}
Let $\vx_1,\dots,\vx_k$ be vectors of norm at least $1$ in a normed space such that \[\Bignorm{\sum_{i=1}^k \vx_i}\leq 1\text{.}\]
Then for each $i\in[k]$ there exists $j\in[k]$ such that $\norm{\vx_i-\vx_j}\geq 1$.
\end{lemma}
\begin{proof}
By the triangle inequality and the hypotheses,
\begin{align*}
k\leq \norm{k\vx_i} &= \Bignorm{\sum_{j=1}^k \vx_j + \sum_{j=1}^k(\vx_i-\vx_j)}\leq \Bignorm{\sum_{j=1}^k \vx_j} + \sum_{j=1}^k\norm{\vx_i-\vx_j}\\
&\leq 1+\sum_{j=1}^k\norm{\vx_i-\vx_j} = 1+\sum_{\substack{j=1\\ j\neq i}}^k\norm{\vx_i-\vx_j}.
\end{align*}
The average distance between $\vx_i$ and the other points is then bounded below:
\[ \frac{1}{k-1}\sum_{\substack{j=1\\ j\neq i}}^k\norm{\vx_i-\vx_j} \geq 1,\]
which implies that $\norm{\vx_i-\vx_j}\geq 1$ for some $j\neq i$.
\end{proof}

\begin{proof}[Proof of Theorem~\ref{bmthm}]
Let $V=\setbuilder{\vx_i}{i\in[m]}\subset X^d$ be a $k$-collapsing family with each $\norm{\vx_i}\geq 1$.
Define a graph $G$ on $V$ by joining $\vx_i$ and $\vx_j$ if $\norm{\vx_i-\vx_j} < 1$.
By Lemma~\ref{triangle}, the maximum degree $\Delta$ of $G$ is at most $k-2$.
By the Hajnal-Szemer\'edi Theorem, $G$ has an equitable $k$-colouring.
This gives a partition $I_1,\dots,I_k$ of $[m]$ such that each $\card{I_t}\in\set{q,q+1}$, where $q:=\lfloor m/k\rfloor$, and such that $\norm{\vx_i-\vx_j}\geq 1$ whenever $i,j$ are distinct elements from the same $I_t$.
For each $t\in[k]$ let
\[S_t:=\bigcup_{j\in I_t} B\Bigl(\vx_j,1/2\Bigr). \]
Then
\begin{equation}\label{volSi}
\vol{S_t}=(1/2)^d\card{I_t}\vol{B}.
\end{equation}
By the $k$-collapsing property,
\begin{equation}\label{volsumSi}
\frac{1}{k}(S_1+\dots+ S_k)\subseteq B\left(\vo,\frac{1}{2}+\frac{1}{k}\right).
\end{equation}
Substitute \eqref{volSi} and \eqref{volsumSi} into the $k$-fold Brunn-Minkowski inequality
\[ \prod_{t=1}^k\vol{S_t}^{1/k}\leq\Bigvol{\frac{1}{k}(S_1+\dots+S_k)},\]
to obtain
\[\biggl(\prod_{t=1}^k\card{I_t}\biggr)^{1/k}\leq\biggl(1+\frac{2}{k}\biggr)^d.\]
Set $r:=m-kq$.
There are $r$ sets $I_t$ of cardinality $q+1$ and $k-r$ of cardinality $q$.
Therefore,
\begin{equation}\label{betterineq}
\biggl(\Bigl(\frac{m-r}{k}+1\Bigr)^r\Bigl(\frac{m-r}{k}\Bigr)^{k-r}\biggr)^{1/k}\leq\biggl(1+\frac{2}{k}\biggr)^d.
\end{equation}
Instead of minimising the left-hand side over all $r\in\set{0,1,\dots,k-1}$, we weaken it to \[ \frac{m-r}{k}\leq\biggl(1+\frac{2}{k}\biggr)^d,\]
to obtain
\[m\leq k\biggl(1+\frac{2}{k}\biggr)^d + r\leq k\biggl(1+\frac{2}{k}\biggr)^d +k-1.\qedhere\]
\end{proof}
By taking more care in minimising the left-hand side of \eqref{betterineq} it is possible to find a slightly better upper bound.
However, the inequality $\Cupper(k,d)\leq k\left(1+\frac{2}{k}\right)^d$ cannot be obtained from~\eqref{betterineq}.
For example, the values $d=4$, $m=19$, $k=6$ satisfy \eqref{betterineq}, but not $m\leq k\left(1+\frac{2}{k}\right)^d$.
(Of course $\Cupper(6,4)=8$ by Theorem~\ref{newthm}.)

\section{The Brunn-Minkowski inequality and Carath\'eodory's theorem}\label{sectionCara}
In this section we consider $k$-collapsing sets when $k\ggg d$ as $d\to\infty$.
We use the Brunn-Minkowski inequality in much the same way as before, but now coupled with Carath\'eodory's theorem from combinatorial convexity.

\theoremstyle{plain}
\newtheorem*{caratheorem}{Carath\'eodory's Theorem}
\begin{caratheorem}
Suppose that $\vp$ is in the convex hull of a family $\setbuilder{\vx_i}{i\in I}$ of points in $\bR^d$.
 Then $\vp\in\conv\setbuilder{\vx_i}{i\in J}$ for some $J\subseteq I$ with $\card{J}\leq d+1$.
\end{caratheorem}
\noindent Carath\'eodory's theorem is used to prove the following auxiliary result.
The technique is very similar to an argument in \cite{Sw07b} that bounds the number of vertices of an edge-antipodal polytope.
\begin{lemma}\label{diameter}
Let $d\geq 2$, $n\geq 1$ and $\setbuilder{\vx_i}{i\in[n]}\subset X^d$ be such that $\norm{\vx_i}\geq 1$ for each $i\in[n]$ and
\begin{equation}\label{diameq}
\diam\setbuilder{\vx_i}{i\in[n]}<1+1/d\text{.}
\end{equation}
Then
\begin{equation}\label{centroidineq}
\Bignorm{\frac{1}{n}\sum_{i=1}^n\vx_i}> 1/d^2\text{.}
\end{equation}
\end{lemma}
\begin{proof}
Let $P:=\conv\setbuilder{\vx_i}{i\in[n]}$.
By convexity, the centroid $\frac{1}{n}\sum_{i=1}^n\vx_i$ is in $P$.
Choose $\vp\in P$ of minimum norm.
It is sufficient to prove that $\norm{\vp}>1/d^2$.
Suppose first that $\vp=\vo$.
Then by Carath\'eodory's Theorem, 
$\vo=\sum_{i\in J}\lambda_i\vx_{i}$ where $J\subseteq[n]$, $\card{J}\leq d+1$, $\lambda_i\geq 0$ for each $i\in J$, and $\sum_{i\in J}\lambda_i=1$.
Note that $\card{J}\geq 2$.
For any $j\in J$,
\[ -\vx_{j} = \sum_{i\in J\setminus\set{j}}\lambda_i(\vx_{i}-\vx_{j}),\]
hence, by the triangle inequality,
\[ 1\leq\sum_{i\in J\setminus\set{j}}\lambda_i\norm{\vx_{i}-\vx_{j}}\leq\sum_{i\in J\setminus\set{j}}\lambda_i\diam P=(1-\lambda_{j})\diam P\text{.}\]
Summing over all $j\in J$, we obtain $\card{J}\leq (\card{J}-1)\diam P$ and \[\diam P\geq\frac{\card{J}}{\card{J}-1}\geq\frac{d+1}{d}.\]
However, \[\diam P=\diam\setbuilder{\vx_i}{i\in[n]}<1+1/d\] by assumption, a contradiction.
It follows that $\vp\neq\vo$, hence $\vp$ is in some facet of $P$.
We apply Carath\'eodory's Theorem to the affine span of this facet, which is of dimension $<d$:
\[ \vp=\sum_{i\in J}\lambda_i\vx_{i}\quad\text{where $J\subseteq[n]$, $\card{J}\leq d$, $\lambda_i\geq 0$ for each $i\in J$, and $\sum_{i\in J}\lambda_i=1$.}\]
If $\card{J}=1$ then $\vp=\vx_i$ for some $i\in[n]$ and $\norm{\vp}\geq 1>1/d^2$.
Thus, without loss of generality we assume that $\card{J}\geq 2$.
It follows that for each $j\in J$,
\[ \vp-\vx_{j} = \sum_{i\in J\setminus\set{j}}\lambda_i(\vx_{i}-\vx_{j})\]
and, again by the triangle inequality,
\begin{align*}
1-\norm{\vp} &\leq\norm{\vx_{j}}-\norm{\vp}\leq\norm{\vp-\vx_{j}}
\leq\sum_{i\in J\setminus\set{j}}\lambda_i\norm{\vx_{i}-\vx_{j}}\\
&\leq\sum_{i\in J\setminus\set{j}}\lambda_i\diam P = (1-\lambda_j)\diam P\text{.}
\end{align*}
Sum over all $j\in J$ to obtain (since $\card{J}\geq 2$) that
\[ (1-\norm{\vp})\card{J} \leq (\card{J}-1)\diam P < (\card{J}-1)(1+1/d)\]

and \[ 1-\norm{\vp} < \frac{\card{J}-1}{\card{J}}(1+1/d) \leq \frac{d-1}{d}(1+1/d)=1-1/d^2\text{.}\]
It follows that $\norm{\vp}>1/d^2$.
\end{proof}
The above proof in fact shows that if $\diam\set{\vx_i}=1+1/d-\epsi$ for some $\epsi>0$, then
$\Bignorm{\frac{1}{n}\sum_{i=1}^n\vx_i}\geq 1/d^2+(1-1/d)\epsi$.
This inequality is sharp, at least for $n=d$, as the following example shows.
Let $\ve_1,\dots,\ve_{d+1}$ be the standard unit basis of $\ell_1^{d+1}$ and consider the $d$\dimensional subspace
\[X^d := \setbuilder{(\alpha_1,\dots,\alpha_{d+1})}{\sum_{i=1}^{d}\alpha_i=0}\subset\ell_1^{d+1}.\]
(For instance, $X^2$ is isometric to $\ell_1^2$ and $X^3$ has a double hexagonal pyramid as unit ball.)
Set
\[ \vx_i := \frac12\biggl(\frac{d+1}{d}-\epsi\biggr)\biggl(\ve_i-\frac{1}{d}\sum_{j=1}^d\ve_j\biggr)+\biggl(\frac{1}{d^2}+\Bigl(1-\frac{1}{d}\Bigr)\epsi\biggr)\ve_{d+1}\quad\text{for $i\in[d]$.}\]
Then $\setbuilder{\vx_i}{i\in[d]}$ satisfies the hypotheses of Lemma~\ref{diameter}: $\norm{\vx_i}_1=1$ for all $i\in[d]$, $\diam\setbuilder{\vx_i}{i\in[d]}=1+1/d-\epsi$, and $\Bignorm{\frac{1}{d}\sum_{i=1}^d\vx_i}_1=1/d^2+(1-1/d)\epsi$.

A slight modification of this example also shows that the right-hand side of \eqref{diameq} cannot be increased:
There exist $d$\dimensional Minkowski spaces with $d+1$ unit vectors $\vx_1,\dots,\vx_{d+1}$ such that $\diam\set{\vx_i}=1+1/d$ although $\sum_{i=1}^{d+1}\vx_i=\vo$.
Let
\[Y^d := \setbuilder{(\alpha_1,\dots,\alpha_{d+1})}{\sum_{i=1}^{d+1}\alpha_i=0}\text{,}\]
also considered as a subspace of $\ell_1^{d+1}$.
(Then $Y^2$ has a regular hexagon as unit ball, and $Y^3$ has a rhombic dodecahedron as unit ball.)
Let \[\vy_i:=\frac{d+1}{2d}\ve_i-\frac{1}{2d}\sum_{j=1}^{d+1}\ve_j\quad(i\in[d+1])\text{.}\]
Then the $\vy_i$ are unit vectors in $Y^d$, $\norm{\vy_i-\vy_j}_1=1+1/d$ for distinct $i,j\in[d+1]$, and $\sum_{i=1}^{d+1}\vy_i=\vo$.

It may seem strange that the centroid of the vectors can jump from the origin to a point bounded away from the origin by a distance of $1/d^2$ when the diameter goes below $1+1/d$.
However, a similar phenomenon occurs even in Euclidean space.
Consider a regular simplex inscribed in the unit sphere of $\ell_2^d$.
Then it is not possible to continuously move the $d+1$ vertices an arbitrarily small distance while remaining on the sphere so as to reduce the diameter of the simplex.
The diameter will increase at first and after it has eventually decreased below the diameter of the original equilateral simplex, the centroid will be bounded away from the origin.

\begin{proof}[Proof of Theorem~\ref{thm:klarge}]
Suppose that $\C_k(X^d)\geq k+2$.
Let $\setbuilder{\vx_i}{i\in[k+2]}\subset X^d$ be a $k$-collapsing collection of vectors of norm at least $1$.
We aim to show that $k=O(d^{d+2})$.

Let $\vs:=\sum_{i=1}^{k+2}\vx_i$.
The $k$-collapsing condition gives an upper bound to the norm of $\vs$ as follows:
Since \[ \sum_{S\in\binom{[k+2]}{k}}\sum_{i\in S} \vx_i = \binom{k+1}{k-1}\sum_{i=1}^{k+2}\vx_i=\binom{k+1}{k-1}\vs\text{,}\]
the triangle inequality gives
\[ \binom{k+1}{k-1}\norm{\vs}\leq\sum_{S\in\binom{[k+2]}{k}}\Bignorm{\sum_{i\in S} \vx_i}\leq\binom{k+2}{k}\text{,}\]
and
\begin{equation}\label{sbound}
\norm{s}\leq\binom{k+2}{k}/\binom{k+1}{k-1}=1+2/k\text{.}
\end{equation}
Without loss of generality, $\norm{\vx_i}=1$ for some $i\in[k+2]$.
For each $j\in[k+2]\setminus\set{i}$ the $k$-collapsing condition implies that $\norm{(\vs-\vx_i)-\vx_j}\leq 1$, and again by the triangle inequality,
\begin{equation}\label{xbound}
\norm{\vx_j}\leq 1+\norm{\vs}+\norm{\vx_i}\leq 3+2/k\text{.}
\end{equation}
Let $\epsi>0$ (to be fixed later).
Define a graph $G$ on $[k+2]$ by joining $i$ and $j$ whenever $\norm{\vx_i-\vx_j}<\epsi$.
Let $C\subseteq[k+2]$ be the set of all isolated vertices of $G$.
Suppose for the moment that $\card{C}\geq 2$.
Partition $C$ into two parts as equally as possible:
$C=C_1\cup C_2$ with $C_1\cap C_2=\varnothing$ and $\bigabs{\card{C_1}-\card{C_2}}\leq 1$.
Let \[S_t:=\bigcup_{j\in C_t}B(\vx_j,\epsi/2)\quad\text{for }t=1,2\text{.}\]
Then 
\[ \vol{S_t}=\card{C_t}(\epsi/2)^d\vol{B}\text{.}\]
By the $k$-collapsing condition, $S_1+S_2\subseteq B(\vs,1+\epsi)$, which gives
\[\Bigvol{\frac{1}{2}S_1+\frac{1}{2}S_2}\leq \Bigl(\frac{1+\epsi}{2}\Bigr)^d\vol{B}\text{.}\]
By the Brunn-Minkowski inequality,
\[ \Bigvol{\frac{1}{2}S_1+\frac{1}{2}S_2}\geq \vol{S_1}^{1/2}\vol{S_2}^{1/2}=\sqrt{\card{C_1}\cdot\card{C_2}}(\epsi/2)^d\vol{B}\text{.}\]
It follows that \[ \frac{\card{C}-1}{2}<\sqrt{\card{C_1}\cdot\card{C_2}}\leq\Bigl(1+\frac{1}{\epsi}\Bigr)^d \]
and
\begin{equation}\label{Cbound}
\card{C} < 2\Bigl(1+\frac{1}{\epsi}\Bigr)^d+1\text{.}
\end{equation}
This bound clearly also holds if $\card{C}<2$.

Next consider the complement $C':=[k+2]\setminus C$, consisting of the vertices of $G$ of degree at least $1$.
We claim that
\begin{equation}\label{claim}
\diam \setbuilder{\vx_i}{i\in C'} < 1+\epsi\text{.}
\end{equation}
Consider distinct $i, j\in C'$.
There exist $i',j'\in C'$ such that $i'\neq i$, $j'\neq j$, $\norm{\vx_i-\vx_{i'}}<\epsi$ and $\norm{\vx_j-\vx_{j'}}<\epsi$.
Then by the triangle inequality and the $k$-collapsing condition,
\begin{align*}
\norm{2\vx_i-2\vx_j} &= \norm{\vx_i-\vx_{i'}+\vx_i+\vx_{i'}-\vs+\vs-\vx_j-\vx_{j'}+\vx_{j'}-\vx_j}\\
&\leq \norm{\vx_i-\vx_{i'}}+\norm{\vx_i+\vx_{i'}-\vs}+\norm{\vs-\vx_j-\vx_{j'}}+\norm{\vx_{j'}-\vx_j}\\
&< \epsi+1+1+\epsi\text{,}
\end{align*}
which shows~\eqref{claim}.
In order to apply Lemma~\ref{diameter} to $\setbuilder{\vx_i}{i\in C'}$ we set $\epsi=1/d$ and obtain that
\begin{equation}\label{C'}
\Bignorm{\sum_{i\in C'}\vx_i}>\frac{\card{C'}}{d^2}=\frac{k+2-\card{C}}{d^2}\text{.}
\end{equation}
On the other hand, by \eqref{sbound} and \eqref{xbound},
\begin{align*}
\Bignorm{\sum_{i\in C'}\vx_i}&=\Bignorm{\vs-\sum_{i\in C}\vx_i}\leq\norm{\vs}+\sum_{i\in C}\norm{\vx_i}\\
&\leq 1+\frac{2}{k}+\card{C}\left(3+\frac{2}{k}\right)\text{.}
\end{align*}
By \eqref{Cbound} and the choice of $\epsi$, $\card{C}\leq2(d+1)^d$.
Combined with \eqref{C'}, we obtain
\[ \frac{k+2}{d^2}<1+\frac{2}{k}+\card{C}\left(3+\frac{2}{k}+\frac{1}{d^2}\right)=O(d^d)\text{.}\qedhere\]
\end{proof}

\section{Reformulation in terms of matrices}\label{section:matrix}
In this section we reduce the existence of a $d$\dimensional Minkowski space admitting vectors satisfying the $k$-collapsing or strong balancing conditions to the existence of a matrix of rank at least $d$ satisfying certain properties.
As a consequence we show that there is no loss of generality in assuming that the vectors in the definitions of $\C(k,d)$ and $\CB(k,d)$ (Def.~\ref{def13}) are unit vectors.
We also present a general version of a well-known lower bound for the rank of a square matrix in terms of its trace and Frobenius norm.

%

\begin{lemma}\label{normalisation}
Let $2\leq k\leq m-2$.
Suppose that $\set{\alpha_1,\dots,\alpha_m}\subset\bR$ is a $k$-collapsing family of real numbers.
If $\abs{\alpha_i}\geq 1$ for some $i\in[m]$, then $\abs{\alpha_j}\leq 2-\abs{\alpha_i}\leq 1$ for all $j\neq i$.
\end{lemma}
\begin{proof}
Without loss of generality, $\alpha_m\geq 1$.
Let $j\in[m-1]$.
Choose any $I,J\in\binom{[m]}{k}$ such that $I\setminus J=\set{m}$ and $J\setminus I=\set{j}$.
By the $k$-collapsing condition, $\sum_{s\in I} \alpha_s\leq 1$ and $\sum_{s\in J} \alpha_{s}\geq -1$.
Subtract these two inequalities to obtain $\alpha_m-\alpha_j\leq 2$, hence $\alpha_j\geq \alpha_m-2$.

Before proving that $\alpha_j\leq 2-\alpha_m$, we first show that
\[ S:=\setbuilder{s\in[m]}{\alpha_i>0}\]
contains at most $k-1$ elements.
Suppose this is false.
Choose any $I\in\binom{S}{k}$ such that $m\in I$.
By the $k$-collapsing condition,
\[0<\sum_{s\in I\setminus\set{m}}\alpha_s\leq 1-\alpha_m\leq 0,\] a contradiction.
Consequently, \[\card{[m]\setminus(S\cup\set{j})}\geq m-k\geq 2\text{,}\]
and there exist two distinct
indices $i',j'\in[m]\setminus\set{j,m}$ such that $\alpha_{i'}\leq 0$ and $\alpha_{j'}\leq 0$.
Choose any $I,I'\in\binom{[m]}{k}$ such that $I\setminus I'=\set{j,m}$ and $I'\setminus I=\set{i',j'}$.
By the $k$-collapsing condition, $\sum_{s\in I} \alpha_{s}\leq 1$ and $\sum_{s\in I'} \alpha_{s}\geq -1$.
Subtract these inequalities to obtain $\alpha_{j}+\alpha_{m}-\alpha_{i'}-\alpha_{j'}\leq 2$.
Therefore, \[ \alpha_{j} \leq 2-\alpha_{m}+\alpha_{i'}+\alpha_{j'}\leq 2-\alpha_m. \qedhere \]
\end{proof}
Lemma~\ref{normalisation} does not hold if $k=m-1\geq 4$, as shown by
\[(\alpha_1,\dots,\alpha_{m})=\Bigl(\underbrace{\frac{-2}{m-3}}_{\text{$m-2$ times}}, \; \frac{m-1}{m-3}\;, \; \frac{2m-4}{m-3}\; \Bigr).\]
However, it is easily seen that Lemma~\ref{normalisation} holds when $k=m-1\in\set{2,3}$.

\begin{lemma}\label{matrixreduction}
Let $2\leq k<m$ and $d\geq 2$.
Let $X^d$ be a $d$\dimensional Minkowski space, $\vx_1,\dots,\vx_m\in X^d$, and $\vx_i^*\in (X^d)^*$ a dual unit vector of $\vx_i$ for each $i\in[m]$.
Then the $m\times m$ matrix $A=[a_{i,j}]:=[\ipr{\vx_i^*}{\vx_j}]$
has rank at most $d$ and satisfies the following properties:
\begin{gather}
\text{$a_{i,i}\geq 1$ for all $i\in[m]$\; if $\norm{\vx_i}\geq 1$ for all $i\in[m]$,} \label{m1}\\
\left.\begin{split}
\text{$a_{i,i} = 1$ for all $i\in[m]$ and $\abs{a_{i,j}}\leq 1$ for all distinct $i,j\in[m]$}\\ \text{if $\norm{\vx_i}=1$ for all $i\in[m]$,}
\end{split}\quad\right\}\label{m2}\\
\text{each row of $A$ is $k$-collapsing\; if $\setbuilder{\vx_i}{i\in[m]}$ is $k$-collapsing,} \label{m3}\\
\text{and the sum of each row of $A$ is $0$\; if $\sum_{i=1}^m\vx_i=\vo$.} \label{m4}
\end{gather}
Conversely, given any $m\times m$ matrix $A=[a_{i,j}]$ with $\rank A\leq d$, there exists a $d$\dimensional Minkowski space $X^d$ and a family $\setbuilder{\vx_i}{i\in[m]}\subset X^d$ such that
\begin{gather}
\text{$\norm{\vx_i}\geq 1$ for all $i\in[m]$\; if $a_{i,i}\geq 1$ for all $i\in[m]$,} \tag{\ref{m1}\ensuremath{'}}\label{m1'}\\
\left.\begin{split}
\text{$\norm{\vx_i} = 1$ for all $i\in[m]$\; if $a_{i,i} = 1$ for all $i\in[m]$}\\
\text{and $\abs{a_{i,j}}\leq 1$ for all distinct $i,j\in[m]$,}
\end{split}\quad\right\}\tag{\ref{m2}\ensuremath{'}}\label{m2'}\\
\text{$\setbuilder{\vx_i}{i\in[m]}$ is $k$-collapsing\; if each row of $A$ is $k$-collapsing,} \tag{\ref{m3}\ensuremath{'}}\label{m3'}\\
\text{$\sum_{i=1}^m\vx_i=\vo$\; if the sum of each row of $A$ is $0$.} \tag{\ref{m4}\ensuremath{'}}\label{m4'}
\end{gather}
\end{lemma}
\begin{proof}
Assume first that $\vx_1,\dots,\vx_m\in X^d$ with dual vectors $\vx_1^*,\dots,\vx_m^*\in (X^d)^*$ are given, and let $A=[a_{i,j}]:=[\ipr{\vx_i^*}{\vx_j}]$.
The factorisation
\[ A=[\ipr{\vx_i^*}{\vx_j}]_{i,j\in[m]}=\transpose{[\vx_1^*, \dots, \vx_m^*]}[\vx_1,\dots, \vx_m] \]
of $A$ into matrices of rank at most $d$ shows that $A$ has rank at most $d$.

Since $\abs{a_{i,j}}=\abs{\ipr{\vx_i^*}{\vx_j}}\leq\norm{\vx_j}$ and $a_{i,i}=\ipr{\vx_i^*}{\vx_i}=\norm{\vx_i}$, we obtain
\eqref{m1} and \eqref{m2}.
Also, if $I\in\binom{[m]}{k}$ and $\norm{\sum_{j\in I}\vx_j}\leq 1$, then for any $i\in[m]$,
\[\Bigabs{\sum_{j\in I}a_{i,j}}=\Bigabs{\sum_{j\in I}\ipr{\vx_i^*}{\vx_j}}
=\Bigabs{\Bigipr{\vx_i^*}{\sum_{j\in I}\vx_j}}
\leq\Bignorm{\sum_{j\in I}\vx_j}\leq 1,\]
which gives \eqref{m3}.
Similarly, if $\sum_{j=1}^m\vx_j=\vo$, then for any $i\in[m]$,
\[\sum_{j\in I}a_{i,j}
=\Bigipr{\vx_i^*}{\sum_{j=1}^m\vx_j}=\ipr{\vx_i^*}{\vo}=0,\]
which is \eqref{m4}.

Next, assume that an $m\times m$ matrix $A=[a_{i,j}]$ of rank at most $d$ is given.
Let $\vx_j$ be the $j$-th column of $A$, considered as an element of $\ell_\infty^m$.
Let $X^d$ be any $d$\dimensional subspace of $\ell_\infty^m$ that contains $\lin{\setbuilder{\vx_j}{j\in[n]}}$.
(If $d>m$, let $X^d=\ell_\infty^d$ be a superspace of $\ell_\infty^m$.)
Keeping the definition of $\norm{\cdot}_\infty$ in mind, it is easily seen that \eqref{m1'}, \eqref{m2'}, \eqref{m3'}, and \eqref{m4'} all hold.
\end{proof}
\begin{corollary}\label{cor38}
Let $2\leq k<m$ and $d\geq 2$.
There exists a $d$\dimensional Minkowski space that contains a $k$-collapsing \textup{[}and balancing\textup{]} family of $m$ vectors of norm $\geq 1$ iff there exists a $d$\dimensional Minkowski space that contains a $k$-collapsing \textup{[}and balancing\textup{]} family of $m$ unit vectors.
\end{corollary}
\begin{proof}
The case $k=m-1$ is trivial, as there exist $k+1$ unit vectors that sum to $\vo$ if $d\geq 2$.
Thus, we assume that $k\leq m-2$.
Suppose that there exists a $d$\dimensional Minkowski space that contains $k$-collapsing family of $m$ vectors of norm $\geq 1$ [that satisfies the balancing condition].
By the first part of Lemma~\ref{matrixreduction} there exists an $m\times m$ matrix $A=[a_{i,j}]$ of rank at most $d$, such that each row is $k$-collapsing and $a_{i,i}\geq 1$ for each $i\in[m]$ [and each row sums to $0$].
Crucially, by Lemma~\ref{normalisation}, $\abs{a_{i,j}}\leq 1$ for all $j\neq i$.
If we divide row $i$ of $A$ by $a_{i,i}$, for each $i$, we obtain a matrix $\tilde{A}=[\tilde{a}_{i,j}]:=[a_{i,j}/a_{i,i}]$ of the same rank as $A$, with each row $k$-collapsing, $\tilde{a}_{i,i}=1$, and $\abs{\tilde{a}_{i,j}}\leq 1$ for all $i,j$ [and each row sums to $0$].
By the second part of Lemma~\ref{matrixreduction}, there exists a $d$\dimensional Minkowski space that contains a $k$-collapsing family of unit vectors [and also satisfies the balancing condition].
\end{proof}
\begin{lemma}\label{ranklemma}
Let $A=[a_{i,j}]$ be any $n\times n$ matrix with complex entries.
Then
\begin{equation}\label{star}
 \Bigabs{\sum_{i=1}^n a_{i,i}}^2\leq\rank{A}\biggl(\sum_{i=1}^n\sum_{j=1}^n\abs{a_{i,j}}^2\biggr).
\end{equation}
Equality holds in \eqref{star} if and only if $A$ is a normal matrix and all its non-zero eigenvalues are equal.
If $A$ is a real matrix then equality holds in \eqref{star} if and only if $A$ is symmetric and all its non-zero eigenvalues are equal.
\end{lemma}
The special case where $A$ is real and symmetric is an exercise in Bellman \cite[p.~137]{Bellman}.
Various combinatorial and geometric applications 
may be found in \cite{Alon, Alon2, AIN, Alon-Pudlak, BDWY, SS}.
These papers use \eqref{star} only for symmetric matrices.
If $A$ is not symmetric, it is then replaced by $A+\transpose{A}$, of rank at most $2\rank{A}$, 
and we obtain that
\[ \frac{\Bigabs{\sum_i a_{i,i}}^2}{\sum_{i,j}\abs{a_{i,j}}^2} \leq \frac{\Bigabs{\sum_i 2a_{i,i}}^2}{\sum_{i,j}\abs{a_{i,j}+a_{j,i}}^2} \leq 2\rank{A},\]
where the first inequality follows from the Cauchy-Schwarz inequality in the form $\sum_{i,j}a_{i,j}\conj{a_{j,i}}\leq\sum_{i,j}\abs{a_{i,j}}^2$.
Thus we obtain a lower bound for the rank of a non-symmetric $A$ which is weaker than \eqref{star} by a factor of $2$.
This weakening is usually of no concern in applications.
However, in this paper we need the sharp estimate \eqref{star} for general (real) matrices, in order to obtain the sharp and almost sharp estimates in Theorems~\ref{balancedthm}, \ref{rankthm2} and \ref{newthm}.
\begin{proof}[Proof of Lemma~\ref{ranklemma}]
Let the non-zero eigenvalues of $A$ be $\lambda_1,\dots,\lambda_r$.
Since the result is trivial if $\trace{A}=\sum_{i=1}^n a_{i,i} =0$, we may assume without loss of generality that $r\geq 1$.
By the Schur~decomposition of a square matrix with complex entries \cite{Schur} (see also \cite[Theorem~2.3.1]{HJ}) there exists an $n\times n$ unitary matrix $U$ such that $C=[c_{i,j}]:=U^\ast A U$ is upper triangular.
In particular, the eigenvalues of $A$ are the diagonal entries of $C$, and
\begin{equation}\label{one}
r\leq\rank{C}=\rank{A}.
\end{equation}
Also,
\begin{equation}\label{two}
\Bigabs{\sum_{i=1}^n a_{i,i}} =\abs{\trace{A}} =\Bigabs{\sum_{i=1}^r\lambda_i}\leq\sum_{i=1}^r\abs{\lambda_i},
\end{equation}
and
\begin{equation}\label{three}
\sum_{i=1}^n\sum_{j=1}^n\abs{a_{i,j}}^2=\trace{A^\ast A}=\trace{C^\ast C}=\sum_{i=1}^n\sum_{j=1}^n\abs{c_{i,j}}^2\geq\sum_{i=1}^r\abs{\lambda_i}^2.
\end{equation}
(This inequality $\sum_{i}\abs{\lambda_i}^2\leq\sum_{i,j}\abs{a_{i,j}}^2$ is in Schur's paper \cite{Schur}.)
Finally, by the Cauchy-Schwarz inequality,
\begin{equation}\label{four}
\Bigl(\sum_{i=1}^r\abs{\lambda_i}\Bigr)^2\leq r\sum_{i=1}^r\abs{\lambda_i}^2,
\end{equation}
and \eqref{star} follows from \eqref{one}, \eqref{two}, \eqref{three} and \eqref{four}.

Suppose that equality holds in \eqref{star}.
This gives equality in \eqref{one}--\eqref{four}.
Equality in \eqref{four} gives that all $\abs{\lambda_i}$ are equal.
Equality in \eqref{two} implies that all $\lambda_i$ are positive multiples of each other.
Therefore, all $\lambda_i$ are equal.
Equality in \eqref{three} gives that $C$ is a diagonal matrix, hence $A$ is normal.
If $A$ is real we furthermore obtain that the $\lambda_i$ are real, since they are equal and their sum is the real number $\trace{A}$.
Then $C=C^\ast$, hence $\transpose{A}=A^\ast=A$ and $A$ is symmetric.

\smallskip
Conversely, if $A$ is normal, then $C$ is diagonal, and equality holds in \eqref{one} and \eqref{three}.
If all the non-zero eigenvalues of $A$ are equal, equality holds in \eqref{two} and \eqref{four}, and we obtain equality in \eqref{star}.
\end{proof}

\section{A tight upper bound for \texorpdfstring{$\CB_k(X)$}{CB k(X)}}\label{section:tight}
In this section we prove Theorem~\ref{balancedthm} using the tools of Section~\ref{section:matrix}.
To show that $\CB_k(X^d)\leq \max\set{k+1,2d}$ for all $d$\dimensional $X^d$, it is sufficient by Lemmas~\ref{normalisation} and \ref{matrixreduction} to prove that for any $m\times m$ matrix $A=[a_{i,j}]$ of rank at most $d$, such that each row is $k$-collapsing and has sum $0$, each entry $\abs{a_{i,j}}\leq 1$, and each diagonal entry $a_{i,i}=1$, we have that $m\leq 2d$ if $k\leq m-2$.
By Lemma~\ref{ranklemma} it is sufficient to show that $\bigabs{\sum_i a_{i,i}}^2/\sum_{i,j}\abs{a_{i,j}}^2\geq m/2$.
Since $\sum_i a_{i,i}=m$, this is equivalent to $\sum_{i,j}a_{i,j}^2\leq 2m$.
Also, it follows from $a_{i,i}=1$ that it will be sufficient to show that \[\sum_{\substack{j=1\\j\neq i}}^m a_{i,j}^2\leq 1\text{ for each $i\in[m]$.}\]
This is implied by the next lemma, which solves a convex maximisation problem with linear constraints.
\begin{lemma}\label{optimisation1}
Let $k,m\in\bN$ such that $2\leq k\leq m-2$.
Then
\[ \max\setbuilder{\sum_{i=1}^{m-1} \alpha_i^2}{\sum_{i=1}^{m} \alpha_i=0,\; \alpha_m=1,\; \text{$\setbuilder{\alpha_i}{i\in[m]}$ is $k$-collapsing}} =1\text{.}\]
The maximum value $\sum_{i=1}^{m-1}\alpha_i^2=1$ is attained under these constraints only if for some $j\in [m-1]$, $\alpha_j=-1$ and $\alpha_i=0$ for all $i\in[m-1]\setminus\set{j}$.
\end{lemma}
\begin{proof}
Since $\sum_{i=1}^m\alpha_i=0$, the family $\setbuilder{\alpha_i}{i\in[m]}$ is $k$-collapsing iff it is $(m-k)$-collapsing.
Thus, without loss of generality, $k\leq m/2$.

The $k$-collapsing and balancing conditions imply the following constraints in the variables $\alpha_1,\dots,\alpha_{m-1}$:
\begin{equation}\label{constraint1gen}
\sum_{i\in I} \alpha_i\leq 0\text{ for all }I\in\binom{[m-1]}{k-1}
\end{equation}
and
\begin{equation}\label{constraint2}
\sum_{i=1}^{m-1} \alpha_i=-1.
\end{equation}
Since these constraints, as well as the objective function $f(\alpha_1,\dots,\alpha_{m-1}):=\sum_{i=1}^{m-1}\alpha_i^2$
are symmetric in the variables $\alpha_1,\dots,\alpha_{m-1}$,
we may assume without loss of generality that
\begin{equation}\label{ordered}
\alpha_1\geq \alpha_2\geq \dots \geq \alpha_{m-1}.
\end{equation}
Then \eqref{constraint1gen} becomes equivalent to the single inequality 
\begin{equation}\label{constraint1}
\sum_{i=1}^{k-1}\alpha_i\leq 0.
\end{equation}
By Lemma~\ref{normalisation}, all $\abs{\alpha_i}\leq 1$, and it follows that the $m-1$ linear inequalities in \eqref{ordered} and \eqref{constraint1} define a polytope $P$ in the hyperplane $H$ of $\bR^{m-1}$ defined by \eqref{constraint2}.
The convex function $f$ attains its maximum on $P$ at a vertex of $P$.
Since the point in $(\alpha_1,\dots,\alpha_{m-1})\in\bR^{m-1}$ with coordinates
\[ \alpha_i=\frac{-2i}{m(m-1)},\quad i\in[m-1] \]
satisfies \eqref{constraint2}, as well as \eqref{ordered} and \eqref{constraint1} with strict inequalities (as well as \eqref{constraint2}), $P$ has non-empty interior in $H$.
It follows that $P$ is an $(m-2)$\dimensional simplex, and it is easy to calculate its $m-1$ vertices, as follows.

\subsection*{Case I}
If $\alpha_1=\dots=\alpha_{m-1}$ then \eqref{constraint2} gives
\[(\alpha_1,\dots,\alpha_{m-1})=\Bigl(\underbrace{\frac{-1}{m-1},\dots,\frac{-1}{m-1}}_{\text{$m-1$ times}}\Bigr)\]
and $f(\alpha_1,\dots,\alpha_{m-1})=1/(m-1)<1$.

\subsection*{Case II}
If $\alpha_1=\dots=\alpha_t$ and $\alpha_{t+1}=\dots=\alpha_{m-1}$ for some $t\in[m-2]$, and $\sum_{i=1}^{k-1}\alpha_i=0$, we distinguish between two subcases:
\subsection*{Subcase II.i}
$t\leq k-1$. Then solving these equations with \eqref{constraint2} gives
\[(\alpha_1,\dots,\alpha_{m-1})=\Bigl(\underbrace{\frac{k-1-t}{t(m-k)},\dots,\frac{k-1-t}{t(m-k)}}_{\text{$t$ times}},\underbrace{\vphantom{\frac{k-1-t}{t(m-k)}}\frac{-1}{m-k},\dots,\frac{-1}{m-k}}_{\text{$m-1-t$ times}}\Bigr)\]
and
\begin{align*}
f(\alpha_1,\dots,\alpha_{m-1}) &=\frac{1}{t}\frac{(k-1)^2}{(m-k)^2}+\frac{m-2k+1}{(m-k)^2}\\
&\leq \frac{(k-1)^2+m-2k+1}{(m-k)^2}\quad\text{(since $t\geq 1$)}\\
&\leq \frac{(m/2-1)^2+m-2k+1}{(m/2)^2}\quad\text{(since $2k\leq m$)}\\
&= \frac{(m/2)^2-2k+2}{(m/2)^2} < 1.
\end{align*}
\subsection*{Subcase II.ii}
$t\geq k$. Then
\[(\alpha_1,\dots,\alpha_{m-1})=\Bigl(\underbrace{\vphantom{\frac{-1}{m-1-t}}0,\dots,0}_{\text{$t$ times}},\underbrace{\frac{-1}{m-1-t},\dots,\frac{-1}{m-1-t}}_{\text{$m-1-t$ times}}\Bigr)\]
and
\[ f(\alpha_1,\dots,\alpha_{m-1})=\frac{1}{m-1-t}\leq 1\]
with equality if and only if $t=m-2$, and then \[(\alpha_1,\dots,\alpha_{m-1})=(0,\dots,0,-1)\text{.}\]
This shows that the maximum of $f$ on $P$ is $1$, attained at only one point if the coordinates are in decreasing order.
\end{proof}

\begin{proof}[Proof of Theorem~\ref{balancedthm}]
By Proposition~\ref{inf}, $\CB_k(\ell_\infty^d)=\max\set{k+1,2d}$.
In fact, if $k\leq 2d$, given any norm with unit vector basis $\set{\ve_1,\dots,\ve_d}$, the family $\setbuilder{\pm \ve_i}{i\in[d]}$ is $k$-collapsing if $\sum_{i\in I}\ve_i$ is contained in the unit ball for all $I\subseteq[d]$ with $\card{I}\leq k$.
Any $\vo$-symmetric convex body $C$ that satisfies
\[ P_k:=\conv\setbuilder{\pm\sum_{i\in I}\ve_i}{I\subseteq[d], \card{I}\leq k}\subseteq C\subseteq[-1,1]^d\]
is the unit ball of a norm $\norm{\cdot}_C$ such that $\setbuilder{\pm\ve_i}{i\in[d]}$ is $k'$-collapsing in the norm $\norm{\cdot}_C$ for all $k'=2,\dots,k$, with $\norm{\ve_i}_C=1$.
When $k<d$, $P_k$ is a proper subset of $[-1,1]^d$ and we obtain infinitely many unit balls $C$.
When $k\geq d$, $P_k=[-1,1]^d$ and we obtain the unique norm $\norm{\cdot}_\infty$ assuming that the $k$-collapsing family is of the form $\setbuilder{\pm \ve_i}{i\in[d]}$ where $\set{\ve_1,\dots,\ve_d}$ is a unit basis.
We will next show that if $m\geq k+2$, then a $k$-collapsing, strongly balancing family of vectors of norm at least $1$ has size at most $2d$, and when it has size $2d$, it is indeed made up of a unit basis and its negative.

Let $\setbuilder{\vx_i}{i\in[m]}$ be $k$-collapsing and strongly balancing with each $\norm{\vx_i}\geq 1$.
For each $\vx_i$, let $\vx_i^*\in X^*$ be a dual unit vector.
By Lemma~\ref{matrixreduction}, $A=[a_{ij}]:=[\ipr{\vx_i^*}{\vx_j}]$ is an $m\times m$ matrix of rank at most $d$, each row is $k$-collapsing, each diagonal element is $\geq 1$, and each row sum is $0$.
We will show that $\rank{A}\leq m/2$, with equality implying that, after some permutation of the $\vx_i$,
\begin{equation}\label{matrixA}
A=
\begin{bmatrix*}[r]
 1 & -1 &  0 &  0 & 0 & \cdots & 0 &  0 &  0 \\
-1 &  1 &  0 &  0 & 0 & \cdots & 0 &  0 &  0 \\
 0 &  0 &  1 & -1 & 0 & \cdots & 0 &  0 &  0 \\
 0 &  0 & -1 &  1 & 0 & \cdots & 0 &  0 &  0 \\
\vdots & \vdots & \vdots & \vdots & \vdots & \ddots & \vdots & \vdots & \vdots \\
 0 &  0 &  0 &  0 & 0 & \cdots & 0 &  1 & -1 \\
 0 &  0 &  0 &  0 & 0 & \cdots & 0 & -1 &  1
\end{bmatrix*}.
\end{equation}
By Lemma~\ref{normalisation}, $\abs{a_{i,j}}\leq 1$ for all distinct $i,j$, and it follows that the matrix $\tilde{A}=[\tilde{a}_{i,j}]:=[a_{i,j}/a_{i,i}]$ formed by dividing each row of $A$ by $a_{i,i}$ has the same rank as $A$, and its rows are still $k$-collapsing and sum to $0$.
By Lemma~\ref{optimisation1}, $\sum_{j=1}^m\tilde{a}_{i,j}^2\leq 2$ for all $i\in[m]$, and by Lemma~\ref{ranklemma},
\[ d\geq \rank{A}=\rank{\tilde{A}}\geq \frac{m^2}{2m}=\frac{m}{2}.\]
This shows that $m\leq 2d$.
Suppose now that $m=2d$.
Then $\rank{A}=\rank{\tilde{A}}=d$, by Lemma~\ref{ranklemma} $\tilde{A}$ is symmetric,  and by Lemma~\ref{optimisation1} each row of $\tilde{A}$ has a $1$ on the diagonal, a $-1$ at some non-diagonal entry, and $0$s everywhere else.
Thus $\tilde{A}=I-P$, where $P$ is a symmetric permutation matrix.
The associated permutation must be an involution.
Therefore, after some permutation of the coordinates, $\tilde{A}$ is as in \eqref{matrixA}.
Since $\tilde{A}$ has an off-diagonal entry of absolute value $1$ in each column,
each $a_{i,i}=1$, hence $A=\tilde{A}$ and $\norm{\vx_i}=1$ for all $i\in[m]$.
Since $A=\transpose{[\vx_1^*\dots\vx_{2d}^*]}[\vx_1\dots\vx_{2d}]$ and the submatrix of $A$ consisting of odd rows and columns is the $d\times d$ identity matrix, it follows that $\set{\vx_1,\vx_3,\dots,\vx_{2d-1}}$ and $\set{\vx_1^*,\vx_3^*,\dots,\vx_{2d-1}^*}$ are bases of $X$ and $X^*$, respectively.
Since $\ipr{\vx_i^*}{\vx_1}=\ipr{\vx_i^*}{\vx_2}=0$ for all $i\geq 3$,
\[\vx_1,\vx_2\in\bigcap_{j=2,\dots,d}\ker \vx_{2j-1}^*,\]
which is a one\dimensional subspace of $X$.
Therefore, $\vx_1=-\vx_2$.
Similarly, $\vx_{2j-1}=-\vx_{2j}$ for all $j\in[d]$.
This proves the theorem.
\end{proof}

\section{Tight and almost tight upper bounds for \texorpdfstring{$\C_k(X)$}{C k(X)}}\label{section:tight2}
We now consider the $k$-collapsing condition without any balancing condition.
As in the previous section we solve a convex optimisation problem.
This case is more complicated and our results are only partial.

\begin{lemma}\label{optimisation2}
Let $k,m\in\bN$ be such that $2\leq k\leq m-2$.
Then
\begin{align*}
&\max\setbuilder{\sum_{i=1}^{m-1} \alpha_i^2}{\alpha_m=1,\; \text{$\setbuilder{\alpha_i}{i\in[m]}$ is $k$-collapsing}}\\[5pt]
&\left\{\begin{array}{l}
=\max\set{\dfrac{m-1}{k^2},1,\dfrac{(k-2)^2+m-2}{k^2}} \quad \text{if $k <2m/3$,}\\[10pt]
\leq\max\set{\dfrac{m-1}{k^2},1,\dfrac{(k-2)^2+m-2}{k^2},\dfrac{(k-1)^2}{4(m-k-1)(2k-m)(m-k)}}\\[10pt]
\qquad\text{if $k \geq 2m/3$,}
\end{array}\right.\\[5pt]
&\begin{cases}
\displaystyle=\max\set{\frac{m-1}{4},1} & \text{if $k=2$,}\\[10pt]
\displaystyle=\frac{(k-2)^2+m-2}{k^2} & \text{if $3\leq k\leq\frac{m+2}{4}$,}\\
\displaystyle=1 & \text{if $\frac{m+2}{4}\leq k < \frac{2m}{3}$, $k\geq 3$,}\\
\displaystyle\leq\max\set{1,\frac{(k-1)^2}{4(m-k-1)(2k-m)(m-k)}} & \text{if $k\geq 2m/3$, $k\geq 3$.}
\end{cases}
\end{align*}
\end{lemma}
\begin{proof}
Because the $k$-collapsing condition on $\setbuilder{\alpha_i}{i\in[m]}$ and the objective function $f(\alpha_1,\dots,\alpha_{m-1}):=\sum_{i=1}^{m-1}\alpha_i^2$ are symmetric in $\alpha_1,\dots,\alpha_{m-1}$, we may assume without loss of generality that 
\begin{equation}\label{four'}
\alpha_1 \geq \alpha_2\geq\dots\geq \alpha_{m-1}\text{.}
\end{equation}
Then the $k$-collapsing condition implies
\begin{gather}
-1 \leq \alpha_{m-k} + \alpha_{m-k+1} + \dots + \alpha_{m-1} \label{seven'}\\
\intertext{and}
\alpha_1 + \alpha_2 + \dots \alpha_{k-1} \leq 0. \label{eight'}
\end{gather}
We find the maximum of $f$ over the set $\Delta$ of points $(\alpha_1,\dots,\alpha_{m-1})$ that satisfy \eqref{four'}, \eqref{seven'} and \eqref{eight'}.
In the cases of equality in the statement of the lemma, we will obtain points in $\Delta$ that also satisfy the $k$-collapsing condition.
(In fact it can be shown that \eqref{seven'} and \eqref{eight'} are equivalent to the $k$-collapsing condition given that \eqref{four'} holds.)
By Lemma~\ref{normalisation}, $\abs{\alpha_i}\leq 1$ for each $i\in[m-1]$,
hence $\Delta$ is a polytope.
Setting $\alpha_i=-i/km$ for $i\in[m-1]$, we see that \eqref{four'} and \eqref{eight'} are obviously satisfied with strict inequalities, and \eqref{seven'} because
\[ \sum_{i=m-k}^{m-1} \frac{-i}{km} = -1 + \frac{k(k+1)}{2km} > -1\text{.}\]
It follows that
\[ \left(\frac{-1}{km},\frac{-2}{km},\dots,\frac{-(m-1)}{km}\right)\in\bR^{m-1}\]
is an interior point of $\Delta$.
Since \eqref{four'}, \eqref{seven'} and \eqref{eight'} are $m$ inequalities in total, it follows that $\Delta$ is a simplex.
The convex function $f$ attains its maximum at one of the $m$ vertices of $\Delta$, which we calculate next.
We distinguish between the following three cases:

\subsection*{Case I}
Equality in \eqref{four'} and \eqref{seven'}:
\[ \alpha_1=\dots=\alpha_{m-1}\quad\text{and}\quad -1=\alpha_{m-k}+\dots+\alpha_{m-1}\text{.} \]
The vertex is
\[ (\alpha_1,\dots,\alpha_{m-1})=\Bigl(\,\underbrace{\frac{-1}{k},\dots,\frac{-1}{k}}_{\text{$m-1$ times}}\,\Bigr) \text{,}\]
and \[ f(\alpha_1,\dots,\alpha_{m-1})=\frac{m-1}{k^2}\text{.}\]

\subsection*{Case II}
Equality in \eqref{four'} and \eqref{eight'}:
\[ \alpha_1=\dots=\alpha_{m-1}\quad\text{and}\quad \alpha_{1}+\dots+\alpha_{k-1} = 0 \text{.} \]
Then $(\alpha_1,\dots,\alpha_{m-1})=\vo$ and $f(\alpha_1,\dots,\alpha_{m-1})=0<(m-1)/k^2$.

\subsection*{Case III}
For some $t\in[m-2]$,
\[ \alpha_1 = \dots = \alpha_t =: a\quad \text{and}\quad \alpha_{t+1} = \dots = \alpha_{m-1} =: b \]
and equality in \eqref{seven'} and \eqref{eight'}:
Equality in \eqref{seven'} gives that
\begin{align}
&\text{if } m-k \geq t+1\quad \text{then}\quad b=\frac{-1}{k}\text{;} \tag{\ref{seven'}a}\label{seven'a}\\
&\text{if } m-k \leq t\quad \text{then}\quad (k-m+1+t)a + (m-1-t)b=-1\text{.} \tag{\ref{seven'}b}\label{seven'b}
\end{align}
Independent of these two cases, equality in \eqref{eight'} gives that
\begin{align}
&\text{if } k-1 \leq t\quad \text{then}\quad a=0\text{;} \tag{\ref{eight'}a}\label{eight'a}\\
&\text{if } k-1 \geq t+1\quad \text{then}\quad ta + (k-1-t)b=0\text{.} \tag{\ref{eight'}b}\label{eight'b}
\end{align}
This gives us four subcases, with some being empty, depending on $k$ and $m$.
\subsection*{Subcase III.i}
If $k-1\leq t\leq m-k-1$, then by \eqref{seven'a} and \eqref{eight'a},
\[ (\alpha_1,\dots,\alpha_{m-1}) = \Bigl(\underbrace{0,\dots,0}_{\text{$t$ times}},\underbrace{\frac{-1}{k},\dots,\frac{-1}{k}}_{\text{$m-1-t$ times}}\Bigr) \]
and
\[ f(\alpha_1,\dots,\alpha_{m-1}) = \frac{m-1-t}{k^2} \leq \frac{m-k}{k^2} < \frac{m-1}{k^2}\text{.} \]
This case occurs only if $2k\leq m$.

\subsection*{Subcase III.ii}
If $\max\set{k-1,m-k}\leq t$, then by \eqref{seven'b} and \eqref{eight'a},
\[ (\alpha_1,\dots,\alpha_{m-1}) = \Bigl(\underbrace{0,\dots,0}_{\text{$t$ times}},\underbrace{\frac{-1}{m-1-t},\dots,\frac{-1}{m-1-t}}_{\text{$m-1-t$ times}}\Bigr) \]
and
\[ f(\alpha_1,\dots,\alpha_{m-1}) = \frac{1}{m-1-t} \leq 1\text{,} \]
with equality if $t=m-2$. 
This case always occurs.

\subsection*{Subcase III.iii}
If $t\leq\min\set{k-2,m-k-1}$ (which occurs only if $k\geq 3$), then by \eqref{seven'a} and \eqref{eight'b},
\[ (\alpha_1,\dots,\alpha_{m-1}) = \Bigl(\underbrace{\frac{k-1-t}{kt},\dots,\frac{k-1-t}{kt}}_{\text{$t$ times}},\underbrace{\frac{-1}{k},\dots,\frac{-1}{k}}_{\text{$m-1-t$ times}}\Bigr) \]
and
\begin{align*}
f(\alpha_1,\dots,\alpha_{m-1}) &= \frac{1}{k^2}\left(\frac{(k-1)^2}{t}-2k+1+m\right)\\
&\leq \frac{1}{k^2}\left((k-1)^2-2k+1+m\right)\\
&= \frac{(k-2)^2+m-2}{k^2} =: g(k,m) \text{.}
\end{align*}
Note that $g(k,m)\geq\frac{m-1}{k^2}$ (equality iff $k=3$). 
Also, $g(k,m)\leq 1$ iff $k\geq (m+2)/4$.

\subsection*{Subcase III.iv}
If $m-k\leq t\leq k-2$ (which occurs only if $2k\geq m+2$ and $k\geq 4$), then we solve \eqref{seven'b} and \eqref{eight'b} to obtain
\[ a = \frac{k-1-t}{t + (m-1-k)(k-1)} \quad\text{and}\quad b = \frac{-t}{t+(m-1-k)(k-1)}\text{.} \]
This gives the vertex as
\[ (\alpha_1,\dots,\alpha_{m-1}) = \Bigl(\underbrace{\frac{k-1-t}{t+(m-1-k)(k-1)}}_{\text{$t$ times}}, \underbrace{\frac{-t}{t+(m-1-k)(k-1)}}_{\text{$m-1-t$ times}}\Bigr) \]
and
\[ f(\alpha_1,\dots,\alpha_{m-1}) = \frac{(m-2k+1)t^2 + (k-1)^2t}{(t+(m-1-k)(k-1))^2} =: s_{k,m}(t) \text{.} \]
We now determine 
\[ h(k,m) := \max \setbuilder{s_{k,m}(t)}{t\in[m-k,k-2]\vphantom{\sum}} \text{.} \]
Since this maximum could occur in the interior of the interval $[m-k,k-2]$, and the value of $t$ where the maximum occurs might not be integral, we settle for determining the maximum of $s_{k,m}(t)$ over all real values of $t\in[m-k,k-2]$.
Thus $h(k,m)$ will only be an upper bound for the maximum of $f(\alpha_1,\dots,\alpha_{m-1})$ on the vertices of $\Delta$ falling under this subcase.
A calculation shows that $s_{k,m}'(t)\geq 0$ iff
\[ t \leq \frac{(k-1)^2(m-k-1)}{2(2k-m-1)(m-k-1)+k-1} =: t_0 \text{.} \]
We next show that $m-k\leq t_0$ unless $k=4$ and $m=6$.
A calculation shows that
\[ m-k\leq t_0\iff (k-1)^2 \leq \frac{1}{2}(m-k)((k-1)^2-1+(2m-3k)^2) \text{.} \]
Since $m-k\geq 2$, this inequality clearly holds if $2m\neq 3k$, while if $2m=3k$, it is equivalent to 
\[ (k-1)^2\leq \frac{1}{4}k((k-1)^2-1),\]
which holds if $k\geq 5$, but not if $k=4$.
However, in that case $(k,m)=(4,6)$ and $m-k=k-2$.

Next we show that if $k\geq 2m/3$ then $t_0 < k-2$, and if $k<2m/3$ then $t_0 > k-2$.
A calculation gives that
\[ t_0\leq k-2\quad\iff\quad 0\leq (k-2)(m-k)(3k-2m) + 2k-m-1\text{.}\]
Since $2k-m-1>0$, we obtain $t_0 < k-2$ if $k\geq 2m/3$.
Otherwise $3k-2m\leq -1$, and
\begin{align*}
&\phantom{{}\leq{}} (k-2)(m-k)(3k-2m) + 2k-m-1\\
&\leq -(k-2)(m-k) + 2k-m-1 = -(k-1)(m-k-1) < 0\text{.}
\end{align*}
It follows that $t_0 > k-2$ if $k<2m/3$.

In summary,
\[ h(k,m)=
\begin{cases}
s_{k,m}(t_0) & \text{if $k\geq 2m/3$ and $(k,m)\neq(4,6)$,}\\
s_{k,m}(k-2) & \text{if $k<2m/3$ or $(k,m)=(4,6)$.}
\end{cases}\]

We next show that $s_{k,m}(k-2)<1$, which means that this subcase is only relevant when $k\geq 2m/3$ and $(k,m)\neq(4,6)$.
Since
\[ s_{k,m}(k-2) = \frac{(m-2k+1)(k-2)^2 + (k-1)^2(k-2)}{(k-2+(m-1-k)(k-1))^2}\text{,}\]
a calculation shows that
\[ s_{k,m}(k-2) < 1\iff m-2k < (k-1)^2\left((m-k)(m-k-1)-1\right)\text{,}\]
which holds since $m-2k < 0$ and $m-k\geq 2$.
Finally we 
calculate
\[s_{k,m}(t_0) = \frac{(k-1)^2}{4(m-k-1)(2k-m)(m-k)}\text{.} \]
This concludes estimating $f$ at the vertices of $\Delta$.
To summarise the above case analysis, we have shown that
\[\max f(\Delta)=\max\set{\frac{m-1}{k^2},1,\frac{(k-2)^2+m-2}{k^2}}\quad\text{if $k < 2m/3$,}\]
and
\[
\max f(\Delta) 
\leq\max\set{\tfrac{m-1}{k^2},1,\tfrac{(k-2)^2+m-2}{k^2},\tfrac{(k-1)^2}{4(m-k-1)(2k-m)(m-k)}}\;\text{if $k \geq 2m/3$.}
\] 
The remaining claims of the lemma are now easily checked.
\end{proof}

\begin{proof}[Proof of Theorem~\ref{rankthm2}]
\quad (1) Let $\sqrt{d}<k\leq (d+1)/2$.
Suppose that there exist $m>2d(1+\frac{d-2k+1}{k^2-d})$ vectors of norm $\geq 1$ satisfying the $k$-collapsing condition;
equivalently, an $m\times m$ matrix $A$ of rank $\leq d$ with $1$s on the diagonal and such that each row satisfies the $k$-collapsing condition.
Since $m>2d\geq 2k-1$, we have $k<(m+2)/4$,
and by Lemma~\ref{optimisation2} the sum of the squares of the entries in any row of $A$ is $\leq 1+\frac{(k-2)^2+m-2}{k^2}=2+\frac{m-4k+2}{k^2}$.
By Lemma~\ref{ranklemma},
\[d\geq\rank A\geq\frac{m^2}{m\left(2+\frac{m-4k+2}{k^2}\right)}=\frac{mk^2}{2k^2+m-4k+2}\text{.}\]
Solving for $m$ (and taking note that $k>\sqrt{d}$) we obtain
\[ m\leq \frac{2d(k-1)^2}{k^2-d}\text{,}\]
contradicting the assumption on $m$.
This shows that $\Cupper(k,d)\leq\frac{2d(k-1)^2}{k^2-d}$.

\medskip
\noindent(2) In particular we obtain that $\Cupper(k,d)\leq 2d$ when $\sqrt{d} < k \leq (d+1)/2$ if
\[ \frac{2d(k-1)^2}{k^2-d} < 2d+1\text{,}\]
which is equivalent to $k\geq-2d+\sqrt{6d^2+3d+1}$.
It remains to show that $\Cupper(k,d)\leq 2d$ if $(d+1)/2 < k\leq 2d-\sqrt{d/2}$.
Suppose that there exists an $m\times m$ matrix $A$ of rank $\leq d$ with $1$s on the diagonal and such that each row satisfies the $k$-collapsing condition, where $m=2d+1$.
It then follows from $k>(d+1)/2$ that $k>(m+2)/4$.
If furthermore $k<2m/3$ then by Lemmas~\ref{ranklemma} and \ref{optimisation2}, $d\geq\rank A\geq\frac{m^2}{m(1+1)}$ and $m\leq 2d$, a contradiction.
Therefore, $k\geq 2m/3$.
We next show that
\begin{equation}\label{ineq?}
\frac{(k-1)^2}{4(m-k-1)(2k-m)(m-k)} < 1 \text{,}
\end{equation}
which again gives the contradiction $m\leq 2d$ by Lemmas~\ref{ranklemma} and \ref{optimisation2}.

Consider $f(x)=(m-x-1)(2x-m)(m-x)$, $2m/3\leq x\leq m-2$.
Then $f'(x)=(4m-6x)(m-x-1)-2x+m < 0$, and it follows that the left-hand side of \eqref{ineq?} increases with $k$.
It is therefore sufficient to prove \eqref{ineq?} for $k=2d-\sqrt{d/2}$, that is
\[\frac{(2d-\sqrt{d/2}-1)^2}{4\sqrt{d/2}(2d-2\sqrt{d/2}-1)(\sqrt{d/2}+1)}<1\text{.}\]
This is equivalent to $8d\sqrt{d/2}-5d/2-6\sqrt{d/2}-1>0$, which is easily seen to be true.

\medskip
\noindent(3)
Let $d\geq 3$ and $k>2d-\sqrt{d/2}$.
Suppose that there exists an $m\times m$ matrix of rank $\leq d$ with $1$s on its diagonal and each row $k$-collapsing, where $m > k+\frac{1+\sqrt{2d-3}}{2}$.
As before, we aim to find a contradiction using Lemmas~\ref{ranklemma} and \ref{optimisation2}.

Writing $t=m-k$, we have $t>\frac{1+\sqrt{2d-3}}{2}>1$.
It follows that $d<2t^2-2t+2 < 2t^2$, hence $k>2d-\sqrt{d/2}>2d-t$ and $m=k+t\geq 2d+1$.

Since we may assume without loss of generality that \[m = \left\lfloor k+\frac{1+\sqrt{2d-3}}{2}\right\rfloor+1\text{.}\]
Since \[3k>3(2d-\sqrt{d/2})=3d+3(d-\sqrt{d/2})>3d>2d+\sqrt{d/2}\geq 4+\sqrt{d/2},\]
we have $4k-2\geq k+2+\sqrt{d/2}>m$ and 
$k>(m+2)/4$.
By Lemma~\ref{optimisation2}, if $k< 2m/3$ or
\[\frac{(k-1)^2}{4(m-k-1)(2k-m)(m-k)}\leq 1\text{,}\]
then Lemma~\ref{ranklemma} would give $d\geq\frac{m^2}{m(1+1)}$ and $m\leq 2d$, a contradiction.
Therefore, $k\geq 2m/3$ and
\[\frac{(k-1)^2}{4(m-k-1)(2k-m)(m-k)} > 1\text{.}\]
Lemma~\ref{ranklemma} now gives
\[d\geq \frac{m^2}{m\left(1+\dfrac{(k-1)^2}{4(m-k-1)(m-k)(m-2k)}\right)}=\frac{m}{\left(1+\dfrac{(k-1)^2}{4(t-1)t(k-t)}\right)}\text{,}\]
which implies
\begin{equation}\label{ineq!} k+t=m\leq\left(1+\frac{(k-1)^2}{4(t-1)t(k-t)}\right)d \text{.}
\end{equation}
If we set $f(x)=\left(1+\frac{(x-1)^2}{4(t-1)t(x-t)}\right)d-(x+t)$ for $x\geq2d-t+1$,
it follows (since $d<2t^2$, $t\geq 2$, and $k\geq 2m/3$) that
\begin{align*}
f'(x)&=\frac{d}{4(t-1)t}\left(1-\Bigl(\frac{t-1}{x-t}\Bigr)^2\right)-1\\
&< \frac{2t^2}{4(t-1)t}-1=\frac{2-t}{2(t-1)}\leq 0\text{,}
\end{align*}
and $f$ is strictly decreasing.
It follows that since \eqref{ineq!} holds for some $k\geq2d-t+1$, it remains true if we substitute $2d-t+1$ into $k$, that is,
\begin{equation}\label{ineq!!}
2d+1\leq\left(1+\frac{(2d-t)^2}{4(t-1)t(2d-2t+1)}\right)d\text{,}
\end{equation}
which is equivalent to
\begin{equation}\label{ineq!!!}
4(d+1)(t-1)t(2d-2t+1)\leq (2d-t)^2d \text{.}
\end{equation}
We next show that the opposite inequality holds, which gives the required contradiction.
Since $t=m-k=\lfloor\frac{1+\sqrt{2d-3}}{2}\rfloor+1$,
\[ t-1 \leq \frac{1+\sqrt{2d-3}}{2} < t\text{,}\]
or equivalently,
\begin{equation}\label{drange}
2t^2-6t+6\leq d\leq 2t^2-2t+1\text{.}
\end{equation}
It can be checked that
\begin{equation}\label{identity}
\begin{split}
&\mathrel{\phantom{=}} 4(d+1)(t-1)t(2d-2t+1) - (2d-t)^2d\\
&=\phantom{+} (t-1)^3(6t+4)+(t-1)^2-1\\
&\mathrel{\phantom{=}} {} + (2t^2-2t+1-d)\bigl((2d-t-2)^2+12d-4t^2-4t\bigr)\text{.}
\end{split}
\end{equation}
By \eqref{drange}, since $t\geq 2$, 
\[ 12d-4t^2-4t\geq 12(2t^2-6t+6)-4t^2-4t=(5t-9)(4t-8)\geq 0 \text{,}\]
hence \[(2t^2-2t+1-d)((2d-t-2)^2+12d-4t^2-4t)\geq 0\text{.}\]
Substitute this into~\eqref{identity} to obtain
\begin{align*}
&\mathrel{\phantom{=}} 4(d+1)(t-1)t(2d-2t+1) - (2d-t)^2d\\
&\geq (t-1)^3(6t+4)+(t-1)^2-1\\
&> 0\text{,}
\end{align*}
which contradicts \eqref{ineq!!!}.
\end{proof}

\begin{proof}[Proof of Theorem~\ref{newthm}]
Suppose that there exists an $m\times m$ matrix of rank~$\leq d$ with $1$s on its diagonal and with each row $k$-collapsing.
We first treat the case $k=2$.
By Lemmas~\ref{ranklemma} and \ref{optimisation2},
\[\frac{m^2}{m(1+\max\set{1,(m-1)/4})}\leq d\text{.}\]
If the maximum in the denominator equals $1$ then $m\leq 2d$.
Otherwise, $m\leq (1+(m-1)/4)d$ and it follows that $(1-d/4)m\leq 3d/4$.
If $d<4$ then $m\leq 3d/(4-d)$.
In particular, if $d=2$ then $m\leq 3$, and if $d=3$ then $m\leq 9$.
This shows that $\Cupper(2,2)=4$ and $\Cupper(2,3)\leq 9$.

Next assume that $k\geq 3$.
Without loss of generality, $m=k+2>2d$.
We aim for a contradiction.
Clearly, $k=m-2>(m+2)/4$.
If the maximum in Lemma~\ref{optimisation2} equals $1$, Lemma~\ref{ranklemma} gives $m\leq 2d$, a contradiction.
Therefore, $k\geq 2m/3$, the maximum in Lemma~\ref{optimisation2} equals
\begin{equation}\label{mineq}
\frac{(k-1)^2}{4(m-k-1)(2k-m)(m-k)}=\frac{(m-3)^2}{8(m-4)}>1\text{,}
\end{equation}
and by Lemma~\ref{ranklemma},
\begin{equation}\label{k3}
\frac{m^2}{m\left(1+\frac{(m-3)^2}{8(m-4)}\right)}\leq d\text{.}
\end{equation}
By \eqref{mineq}, $m\geq 10$ and $k\geq 8$.
Solving for $m$ in~\eqref{k3} gives
\[ m\leq\frac{d+16+2\sqrt{6d^2-38d+64}}{8-d}\]
if we assume $d<8$.
Since $k=m-2$, we obtain
\[ k \leq \frac{3d+2\sqrt{6d^2-38d+64}}{8-d}\text{.}\]
Keeping in mind that $m=k+2>2d$ and $m\geq 10$, we obtain a contradiction if $d\leq 5$ (and $k\geq 3$); or if $d=6$ and $k\geq 17$; or if $d=7$ and $k\geq 41$.
This proves the theorem.
\end{proof}

\section{Upper bounds using the ranks of Hadamard powers of a matrix}\label{section:power}
The following lemma, used by Alon in \cite{Alon, Alon2}, bounds the ranks of the integral Hadamard powers of a square matrix from above in terms of the rank of the matrix.
It can be used to change a matrix to one that is sufficiently close to the identity matrix so that Lemma~\ref{ranklemma} can give a good bound.

\begin{lemma}[Alon {\cite[Lemma 9.2]{Alon}}]\label{lemma:hpower}
Let $A=[a_{i,j}]$ be an $n\times n$ matrix of rank $d$ (over any field), and let $p\geq 1$ be an integer.
Then the rank of the $p$-th Hadamard power $A^{\odot p}$ satisfies
\[\rank{A^{\odot p}}=\rank{[a_{i,j}^p]} \leq\binom{p+d-1}{p}\text{.}\]
\end{lemma}

In order to use the above lemma in combination with Lemma~\ref{ranklemma} as before, we need to maximise $\sum_i x_i^{2p}$ on the simplex $\Delta$ from the proof of Lemma~\ref{optimisation2}.
Here we restrict the range of $k$ to avoid the difficulties in Case~III.iv in the proof of Lemma~\ref{optimisation2}.
\begin{lemma}\label{optimisation3}
Let $p,k,m\in\bN$ be such that $2\leq k\leq (m+1)/2$.
Then
\begin{align*}
&\max\setbuilder{\sum_{i=1}^{m-1} \alpha_i^{2p}}{\alpha_m=1,\; \setbuilder{\alpha_i}{i\in[m]}\text{ is $k$-collapsing}}\\[5pt]
&=\begin{cases}\displaystyle\max\set{1,\frac{m-1}{k^{2p}}} & \text{if $k=2$,}\\[3mm]
\displaystyle\max\set{1,\frac{(k-2)^{2p}+m-2}{k^{2p}}} & \text{if $k\geq 3$.}
\end{cases}
\end{align*}
\end{lemma}
\begin{proof}
As in the proof of Lemma~\ref{optimisation2} we have to maximise the new objective function $f_p(\alpha_1,\dots,\alpha_{m-1})=\sum_{i=1}^{m-1}x_i^{2p}$ over the same simplex $\Delta$ defined by \eqref{four'}, \eqref{seven'} and \eqref{eight'} as before.
Since $f_p$ is convex, it is again sufficient to calculate the values of $f_p$ on the vertices of $\Delta$.
Using the same case numbering as in the proof of Lemma~\ref{optimisation2}, we obtain the following values:

\subsection*{Case I}
$f_p(\alpha_1,\dots,\alpha_{m-1})=\dfrac{m-1}{k^{2p}}$.

\subsection*{Case II}
$f_p(\alpha_1,\dots,\alpha_{m-1})=0<\dfrac{m-1}{k^{2p}}$.

\subsection*{Subcase III.i}
$f_p(\alpha_1,\dots,\alpha_{m-1}) = \dfrac{m-1-t}{k^{2p}} \leq \dfrac{m-k}{k^{2p}} < \dfrac{m-1}{k^{2p}}$.

\subsection*{Subcase III.ii}
$f_p(\alpha_1,\dots,\alpha_{m-1}) = \dfrac{1}{(m-1-t)^{2p-1}} \leq 1$
with equality iff $t=m-2$.
\subsection*{Subcase III.iii}
\begin{align*}
f_p(\alpha_1,\dots,\alpha_{m-1}) &= \frac{1}{k^{2p}}\left(t\left(\left(\tfrac{k-1}{t}-1\right)^{2p}-1\right)+m-1\right)=:g_p(t)\\
&\leq g_p(1)=\frac{1}{k^{2p}}\left((k-2)^{2p}+m-2\right)
\end{align*}
since $g_p(t)$ is decreasing for $0<t<k-1$.
This case occurs only if $k\geq 3$.

\subsection*{Subcase III.iv}
The case $m-k\leq t\leq k-2$ occurs only if $2k\geq m+2$, which we have assumed to be false.
\end{proof}
\begin{lemma}\label{thm:rank}
If $p\in\bN$ and $k>\binom{d+p-1}{p}^{\frac{1}{2p}}$ then
\[ \Cupper(k,d) < \max\set{\frac{2k^{2p}\binom{d+p-1}{p}}{k^{2p}-\binom{d+p-1}{p}},2k-1}\text{.}\]
\end{lemma}
\begin{proof}
By Lemmas~\ref{normalisation} and \ref{matrixreduction}, there exists an $m\times m$ matrix $A=[a_{i,j}]$ of rank at most $d$, with $1$s on its diagonal, and with each row $k$-collapsing, where $m=\Cupper(k,d)$.
Without loss of generality, $m\geq 2k-1$.
By Lemma~\ref{optimisation3}, for any row $i\in[m]$ of $A^{\odot 2p}$,
\[ \sum_{j=1}^m a_{i,j}^{2p} < 2+\frac{m}{k^{2p}}\]
and by Lemmas~\ref{ranklemma} and \ref{lemma:hpower},
\[\binom{p+d-1}{p}\geq\rank{[a_{i,j}^{2p}]}>\frac{m^2}{m\bigl(2+\frac{m}{k^{2p}}\bigr)}\text{,}\]
from which follows
\[ m < \frac{2k^{2p}\binom{d+p-1}{p}}{k^{2p}-\binom{d+p-1}{p}}\text{.}\qedhere\]
\end{proof}
\begin{proof}[Proof of Theorem~\ref{thm:sqrtd}]
This is just a calculation from Lemma~\ref{thm:rank}.
Since \[ \frac{k^{2p}}{\binom{d+p-1}{p}} > \frac{((p!)^{-1/2p}+\epsi)^{2p}d^p}{\binom{d+p-1}{p}} \xrightarrow{d\to\infty}(1+(p!)^{1/2p}\epsi)^{2p}>1+2p(p!)^{1/2p}\epsi\text{,}\]
it follows that if $d$ is sufficiently large depending on $p$ and $\epsi$, then
\[ \frac{k^{2p}}{\binom{d+p-1}{p}} > 1+p(p!)^{1/2p}\epsi =: 1+\delta\text{,}\]
where $\delta>0$ depends only on $p$ and $\epsi$.
Then \[\binom{d+p-1}{p}^{-1}-k^{-2p}> \frac{\delta}{k^{2p}},\]
and by Lemma~\ref{thm:rank}, (since $\Cupper(k,d)\geq 2d\geq 2k^2>2k-1$)
\[\Cupper(k,d) < \frac{2}{\binom{d+p-1}{p}^{-1}-k^{-2p}} < \frac{2k^{2p}}{\delta}\leq \frac{2d^p}{\delta}.\qedhere\]
\end{proof}
\begin{lemma}\label{lemma:stirling}
Let $n>k\geq 1$ be integers and $\epsi=k/n$.
Then \[\binom{n}{k} < \frac{(\epsi^{-\epsi}(1-\epsi)^{-(1-\epsi)})^n}{\sqrt{2\pi\epsi(1-\epsi)n}}\text{.}\]
\end{lemma}
\begin{proof}
Substitute the Stirling formula in the form $m!=\rme^{\delta_m}(\frac{m}{\rme})^m\sqrt{2\pi m}$, where $\frac{1}{12m+1}<\delta_m<\frac{1}{12m}$ \cite{Robbins} into $\frac{n!}{k!(n-k)!}$ to obtain
\[ \binom{n}{k} < \frac{(\epsi^{-\epsi}(1-\epsi)^{-(1-\epsi)})^n}{\sqrt{2\pi\epsi(1-\epsi)n}}\rme^{\frac{1}{12n}-\frac{1}{12k+1}-\frac{1}{12(n-k)+1}}\text{.}\]
It is easily seen that $\frac{1}{a+b}<\frac{1}{a+1}+\frac{1}{b+1}$ for all $a,b\geq 1$.
In particular, $\frac{1}{12n}<\frac{1}{12k+1}+\frac{1}{12(n-k)+1}$
and the lemma follows.
\end{proof}
\begin{proof}[Proof of Theorem~\ref{rankthm1}]
The function $f(x)=(1+x)^{1/x}(1+1/x)$ is strictly decreasing on $(0,1]$ with $\lim_{x\to0+}f(x)=\infty$ and $f(1)=4$.
Therefore, $\gamma_2=1$ and $(\gamma_k)$ is strictly decreasing.
Since $f(x)<\rme\cdot(1+1/x)$, we have $f(\rme/(k^2-\rme))<k^2$ and $\gamma_k<\rme/(k^2-\rme)$.
Also, since \[\frac{x}{x+1}=1-\frac{1}{x+1}<\rme^{-1/(x+1)}\text{,}\]
it follows that $(1+1/x)^{x+1}>\rme$.
Set $x=k^2/\rme$ to obtain that $f(\rme/k^2)>k^2$ and $\rme/k^2<\gamma_k$.

Let $p:=\lceil \gamma_k d\rceil$ and $\gamma:=p/d$.
Then $\gamma\geq\gamma_k$ and it follows that
\begin{equation}\label{gammaineq}
(1+\gamma)^{1/\gamma}\Bigl(1+\frac{1}{\gamma}\Bigr)\leq k^2\text{.}
\end{equation}
We estimate $\binom{p+d-1}{p}$ as follows:
\begin{align*}
\binom{p+d-1}{p} &=\binom{(1+\gamma)d-1}{\gamma d} =\frac{1}{1+\gamma}\binom{(1+\gamma)d}{\gamma d}\\
&< \frac{\bigl((1+1/\gamma)^\gamma(1+\gamma)\bigr)^d}{\sqrt{2\pi\gamma(1+\gamma)d}}\quad\text{by Lemma~\ref{lemma:stirling}}\\
&\leq \frac{k^{2\gamma d}}{\sqrt{2\pi\gamma(1+\gamma)d}} \quad\text{by \eqref{gammaineq}}\\
&= \frac{k^{2p}}{\sqrt{2\pi\gamma(1+\gamma)d}}\text{.}
\end{align*}
In particular, $\binom{p+d-1}{p} < k^{2p}$
since \[\sqrt{2\pi\gamma(1+\gamma)d}>\sqrt{2\pi\gamma d}=\sqrt{2\pi p}\geq \sqrt{2\pi}>1\text{.}\]
By Lemma~\ref{thm:rank}, either $\Cupper(k,d)<2k-1$ or
\begin{align*}
\Cupper(k,d)
&< \frac{2k^{2p}k^{2p}}{\sqrt{2\pi\gamma(1+\gamma)d}\left(k^{2p}-\dfrac{k^{2p}}{\sqrt{2\pi\gamma(1+\gamma)d}}\right)}\\
&= \frac{2k^{2p}}{\sqrt{2\pi\gamma(1+\gamma)d}-1}\text{.}
\end{align*}
This gives \[\Cupper(k,d)<\max\set{\frac{2}{\sqrt{2\pi}-1}k^{2p}, 2k-1} < 1.33k^{2\gamma_kd+2}\text{.}\]
We now assume that $k<\sqrt{d}$.
Then $\Cupper(k,d)\geq 2d > 2k-2$ and
\begin{align*}
\Cupper(k,d)
&< \frac{2k^{2p}}{\sqrt{2\pi\gamma(1+\gamma)d}-1}\\
&< \frac{2k^{2\gamma_k+2}}{\sqrt{2\pi\gamma_k d}-1}\\
&< \frac{2k^{2\gamma_k+2}}{\sqrt{2\pi(\rme/k^2)d}-1}\text{.}
\end{align*}
Since $\sqrt{2\pi\rme}>3$ and $d/k^2>1$, it follows that $\sqrt{2\pi(\rme/k^2)d}-1>2\sqrt{d/k^2}$ and
$\Cupper(k,d) < k^{3+2\gamma_k d}/\sqrt{d}$.
\end{proof}

\section{Lower bounds}\label{section:lowerbounds}
\begin{lemma}\label{lemma:existence}
Let $k\geq 2$.
Suppose there exist at least $m$ unit vectors $\vu_i\in\ell_2^{d-1}$ such that \[\abs{\ipr{\vu_i}{\vu_j}}\leq\frac{1}{2k+1}\quad\text{for all distinct $i,j$.}\]
Then there exists a $d$\dimensional Minkowski space $X^d$ such that $\C_k(X^d)\geq m$.
If $\abs{\ipr{\vu_i}{\vu_j}}<1/(2k+1)$ for all distinct $i,j$, then $X^d$ can be chosen to be strictly convex and $C^\infty$.
\end{lemma}
\begin{proof}
The construction is similar to the construction in \cite{FLM} of a strictly convex $d$\dimensional space $X^d$ such that $C_2(X^d)\geq 1.02^d$.
The main difference is that we define the unit ball as an intersection of half spaces instead of a convex hull of a finite set of points.

Consider $\ell_2^{d-1}$ to be a hyperplane of $\ell_2^d$ with unit normal $\ve$.
Let $\vx_i=\vu_i+\ve$ and $\vy_i=(1+\frac{1}{2k})\vu_i-\frac{1}{2k}\ve$ for each $i\in[m]$.
Let
\[ B := \setbuilder{\vx\in\ell_2^d}{\abs{\ipr{\vx}{\vy_i}}\leq 1 \text{ for all }i\in[m]}\text{.}\]
If $\lin{\set{\vy_i}}=\bR^d$ then $B$ is bounded and is the unit ball of some norm $\norm{\cdot}_B$.
Otherwise $\set{\vy_i}$ spans a hyperplane with normal $\ve'$, say.
In this case $B$ as defined above is unbounded, so we have to modify it.
Before doing that, we show that $\vx_i\in\boundary B$ and
\[ \sum_{i\in I}\vx_i\in B\quad\text{for all } I\in\binom{[m]}{k}\text{.}\]
Let $i,j\in[m]$. Then
\begin{equation*}
\ipr{\vx_i}{\vy_j}=\Bigl(1+\frac{1}{2k}\Bigr)\ipr{\vu_i}{\vu_j}-\frac{1}{2k}\text{.}
\end{equation*}
In particular, $\ipr{\vx_i}{\vy_i}=1$, and since 
$-\frac{1}{2k+1}\leq\ipr{\vu_i}{\vu_j}\leq\frac{1}{2k+1}$ for distinct $i,j$,
we obtain
\begin{equation}\label{ubounds}
-\frac{1}{k}\leq\ipr{\vx_i}{\vy_j}\leq 0\text{ for distinct $i,j\in[m]$,}
\end{equation}
and it follows that $\vx_i\in\boundary B$.

Next let $I\in\binom{[m]}{k}$ and $i\in[m]$.
We distinguish between two cases, depending on whether $i\in I$ or not.

If $i\notin I$, then by \eqref{ubounds},
\[-1\leq \Bigipr{\sum_{j\in I}\vx_j}{\vy_i}\leq 0\text{.}\]
If $i\in I$, then again by \eqref{ubounds},
\[\frac{1}{k}=1-\frac{k-1}{k}\leq\Bigipr{\sum_{j\in I}\vx_j}{\vy_i}\leq 1\text{.}\]
In both cases we have $\abs{\ipr{\sum_{j\in I}\vx_j}{\vy_i}}\leq 1$ for all $i$, and it follows that $\sum_{j\in I}\vx_j\in B$ for all $I$.
If $\lin{\set{\vy_i}}=\bR^d$, then we have shown that $B$ is the unit ball of a norm $\norm{\cdot}_B$ such that $\set{\vx_i}$ is a $k$-collapsing family of unit vectors in $\bR^d,\norm{\cdot}_B)$.
In the case where $\lin{\set{\vy_i}}$ is a hyperplane with normal $\ve'$, we choose $\lambda>0$ sufficiently large so that $\abs{\ipr{\vx_i}{\ve'}}<\lambda$ for all $i$ and $\abs{\ipr{\sum_{i\in I}\vx_i}{\ve'}}<\lambda$ for all $I\in\binom{[m]}{k}$, and define the required unit ball to be
\[ B := \setbuilder{\vx\in\ell_2^d}{\abs{\ipr{\vx}{\vy_i}}\leq 1 \text{ for all }i\in[m]\text{ and }\abs{\ipr{\vx}{\ve'}}\leq\lambda}\text{.}\]
If $\abs{\ipr{\vu_i}{\vu_j}}<1/(2k+1)$ for distinct $i,j$, then $\abs{\ipr{\sum_{j\in I}\vx_j}{\vy_i}}<1$ for all $i$, and $\sum_{j\in I}\vx_j\in\interior B$ for all $I$.
Also note that no $\vx_j$, $j\neq i$, is on any of the hyperplanes \[\setbuilder{\vx\in\ell_2^d}{\ipr{\vx}{\vy_i}=\pm 1}\text{ or }\setbuilder{\vx\in\ell_2^d}{\ipr{\vx}{\ve'}=\pm\lambda}\text{.}\]
Then a strictly convex and $C^\infty$ norm can be found with unit ball between $\conv\set{\vx_i}$ and $B$ \cite{Ghomi}.
\end{proof}

For a detailed proof of the following lemma, see \cite{Sw04}.
It uses a greedy construction.
\begin{lemma}\label{greedy_eucl}
Let $\delta>0$.
For sufficiently large $d$ depending on $\delta$, there exist $m\geq\left(1+\frac{\delta^2}{2}\right)^d$ unit vectors $\vu_i$ in $\ell_2^{d-1}$ such that $\abs{\ipr{\vu_i}{\vu_j}}<\delta$ for all distinct $i,j$.
\end{lemma}
\begin{proof}[Proof of Theorem~\ref{lbthm}]
Immediate from Lemmas~\ref{lemma:existence} and \ref{greedy_eucl}.
\end{proof}
The following construction was explained to the author by Noga Alon (personal communication).
\begin{lemma}\label{polynomial}
Let $q$ be a prime power and $s\in\bN$ with $s<q$.
Then there exist $q^{s+1}$ unit vectors in $\ell_2^{q^2-q}$ such that the inner product of any two vectors is in the interval $[-\frac{1}{q-1},\frac{s-1}{q-1}]$.
\end{lemma}
\begin{proof}
Let $\CP_s$ be the collection of polynomials over the field of $q$ elements of degree at most $s$:
\[ \CP_s=\setbuilder{p\in\bF_q[x]}{\deg(p)\leq s}\text{.}\]
Then $\card{\CP_s} = q^{s+1}$.
For each $p\in\CP_s$ define a real $q\times q$ matrix $M(p)$ by
\[ M(p)_{i,j} = \begin{cases}
1 & \text{if $p(i)=j$,}\\
-\frac{1}{q-1} & \text{if $p(i)\neq j$.}
\end{cases}\]
These matrices are in the $q^2$\dimensional vector space of all real $q\times q$ matrices with inner product $\ipr{A}{B}=\sum_{i=1}^q\sum_{j=1}^q a_{i,j}b_{i,j}$.

Note that $M(p_1)=M(p_2)$ iff $p_1(x)=p_2(x)$ for all $x\in\bF_q$.
Since $s<q$, all $M(p)$ ($p\in\CP_s$) are distinct
(otherwise $M(p_1)=M(p_2)$ for some $p_1,p_2\in\bF_q[x]$ with $p_1\neq p_2$, and then $p_1-p_2$ would have $q>s\geq\deg(p_1-p_2)$ roots, implying that $p_1-p_2$ is the zero polynomial).
This also shows that two distinct polynomials from $\CP_s$ are equal at at most $s$ points.

Let $p_1,p_2\in\CP_s$ with $p_1$ and $p_2$ not necessarily distinct.
Let $c$ denote the number of points where $p_1$ and $p_2$ coincide.
Then
\begin{align*}
\ipr{M(p_1)}{M(p_2)} &= c -2(q-c)\frac{1}{q-1}+(q^2-2q+c)\frac{1}{(q-1)^2}\\
&= (c-1)\left(\frac{q}{q-1}\right)^2\text{.}
\end{align*}
If $p_1\neq p_2$, then $0\leq c\leq s$ and
\[ -\left(\frac{q}{q-1}\right)^2\leq\ipr{M(p_1)}{M(p_2)}\leq(s-1)\left(\frac{q}{q-1}\right)^2\text{.}\]
On the other hand, since a polynomial coincides with itself at exactly $q$ points, $\ipr{M(p)}{M(p)}=\frac{q^2}{q-1}$.
Thus $\frac{\sqrt{q-1}}{q}M(p)$ has norm $1$, and inner products of distinct $\frac{\sqrt{q-1}}{q}M(p)$ lie in $[\frac{-1}{q-1},\frac{s-1}{q-1}]$.
Since each column of each $M(p)$ sums to $0$, the $M(p)$ lie in a $(q^2-q)$\dimensional subspace of the space of $q\times q$ matrices.
\end{proof}
\begin{proof}[Proof of Theorem~\ref{lbthm2}]
Set $s=c+1$ in Lemma~\ref{polynomial} and then apply Lemma~\ref{lemma:existence}.
\end{proof}

\section*{Acknowledgements}
We thank G\"unter Rote for a valuable hint, Noga Alon for his encouragement and explanation of the construction of Lemma~\ref{polynomial}, and also Imre B\'ar\'any for his encouragement.
Part of this paper was written during a visit to the Discrete Analysis Programme at the Newton Institute in Cambridge in May~2011.


\begin{thebibliography}{99}

\bibitem{Alon} N.~Alon,
\emph{Problems and results in extremal combinatorics. I},
Discrete Math.\ \textbf{273} (2003), 31--53.

\bibitem{Alon2}  N.~Alon,
\emph{Perturbed identity matrices have high rank: proof and applications},
Combin.\ Probab.\ Comput.\ \textbf{18} (2009), 3--15. 

\bibitem{AIN}
N.~Alon, T.~Itoh and T.~Nagatani, \emph{On $(\epsi,k)$-min-wise independent permutations}, Random Structures Algorithms \textbf{31} (2007), 384--389. 

\bibitem{Alon-Pudlak}
N.~Alon and P.~Pudl{\'a}k, 
\emph{Equilateral sets in {$l\sp n\sb p$}},
Geom.\ Funct.\ Anal.\ \textbf{13} (2003), 467--482.

\bibitem{Ball}
K.~Ball,
\emph{An elementary introduction to modern convex geometry}.
In: Flavors of geometry,
Math.\ Sci.\ Res.\ Inst.\ Publ.\ \textbf{31}, Cambridge Univ.\ Press, Cambridge, 1997.
pp.~1--58.

\bibitem{BDWY}
B.~Barak, Z.~Dvir, A.~Wigderson and A.~Yehudayoff, \emph{Rank bounds for design matrices with applications to combinatorial geometry and locally correctable codes}, extended abstract, STOC'11, Proceedings of the 43rd ACM Symposium on Theory of Computing, ACM, New York, 2011. pp.~519--528. 

\bibitem{Bellman}
R.~Bellman, \emph{Introduction to matrix analysis}, Classics in Applied Mathematics, 19. SIAM, Philadelphia, PA, 1997. 

\bibitem{Cockayne}
E.~J. Cockayne, \emph{On the {Steiner} problem}, Canad.\ Math.\ Bull.\ \textbf{10} (1967), 431--450.

\bibitem{CPR}
B.~Codenotti, P.~Pudl\'ak and G.~Resta,
\emph{Some structural properties of low-rank matrices related to computational complexity},
Theoret.\ Comput.\ Sci.\ \textbf{235} (2000), 89--107.

\bibitem{Deschaseaux}
J.-P.~Deschaseaux, \emph{Une caract\'erisation de certains espaces vectoriels norm\'es de dimension finie par leur constante de Macphail},
C. R. Acad.\ Sci.\ Paris S\'er.\ A--B \textbf{276} (1973), A1349--A1351. 

\bibitem{Erdos} P.~Erd\H{o}s,
\emph{Problem 9},
in: Theory of Graphs and Its Applications, (M. Fiedler, ed.), Proceedings of the Symposium held in Smolenice in June 1963, Publishing House of the Czechoslovak Academy of Sciences, Prague, 1964. p.~159.

\bibitem{FV}
C.~Franchetti and G.~F.~Votruba, \emph{Perimeter, Macphail number and projection constant in Minkowski planes}, Boll.\ Un.\ Mat.\ Ital.\ B (5) \textbf{13} (1976), 560--573.

\bibitem{FLM} Z. F\"uredi, J. C. Lagarias and F. Morgan,
\emph{Singularities of minimal surfaces and networks and related extremal problems in Minkowski space}, 
DIMACS Ser.\ Discrete Math.\ Theoret.\ Comput.\ Sci.\, 6,
(J. E. Goodman, R. Pollack and W. Steiger, eds.),
Amer. Math. Soc., Providence, RI, 1991, pp.\ 95--109.

\bibitem{Ghomi}
M.~Ghomi, \emph{Optimal smoothing for convex polytopes}, Bull.\ London Math.\
  Soc.\ \textbf{36} (2004), 483--492.

\bibitem{GS}
P.~M.~Gruber and F.~E.~Schuster, \emph{An arithmetic proof of John's ellipsoid theorem}, Arch.\ Math.\ (Basel) \textbf{85} (2005), 82--88.

\bibitem{HS} A. Hajnal and E. Szemer\'edi,
\emph{Proof of a conjecture of P. Erd\H{o}s},
Combinatorial theory and its applications,
II (Proc.\ Colloq., Balatonf\"ured, 1969),
North-Holland, Amsterdam, 1970. pp.~601--623. 

\bibitem{HJ}
R.~A.~Horn and C.~R.~Johnson, \emph{Matrix Analysis}, Cambridge University Press, Cambridge, 1990.

\bibitem{Katona1}
G.O.H.~Katona, \emph{Inequalities for the distribution of the length of random vector sums}, (in Russian), Teor.\ Verojatnost. i Primenen.\ \textbf{22} (1977), 466--481, translation: Theory Probability Appl.\ \textbf{22} (1977), 450--464.

\bibitem{Katona2}
G.O.H.~Katona, \emph{Sums of vectors and Tur\'an's problem for 3-graphs}, European J. Combin.\ \textbf{2} (1981), 145--154.

\bibitem{Katona3}
G.O.H.~Katona, \emph{``Best'' estimations on the distribution of the length of sums of two random vectors}, Z. Wahrsch.\ Verw.\ Gebiete, \textbf{60} (1982), 411--423.

\bibitem{Katona4}
G.O.H.~Katona, \emph{Probabilistic inequalities from extremal graph results (a survey)},
Random graphs '83 (Pozna\'n, 1983), North-Holland Math.\ Stud.\ \textbf{118}, North-Holland, Amsterdam, 1985.
pp.~159--170.

\bibitem{KMW} G.~O.~H.~Katona, R.~Mayer and W.~A.~Woyczynski,
\emph{Length of sums in a Minkowski space},
In: \emph{Towards a theory of geometric graphs}, (J.~Pach, ed.), Contemp.\ Math.\ \textbf{342}, Amer.\ Math.\ Soc., Providence, RI, 2004. pp.~113--118.

\bibitem{KK} H.~A.~Kierstead and A.~V.~Kostochka,
\emph{A short proof of the Hajnal-Szemer\'edi theorem on equitable colouring},
Combin.\ Probab.\ Comput.\ \textbf{17} (2008), 265--270.

\bibitem{LM} G. R. Lawlor and F. Morgan,
\emph{Paired calibrations applied to soap 
films, immiscible fluids, and surfaces or networks minimizing other norms}, 
Pacific J.\ Math.\ \textbf{166} (1994), 55--83.

\bibitem{Lyusternik} L.~A.~Lyusternik,
\emph{Die Brunn-Minkowskische Ungleichung f\"ur beliebige messbare Mengen},
C. R. (Dokl.) Acad.\ Sci.\ URSS, n. Ser.\ \textbf{3} (1935), 55--58.


\bibitem{MSW}
H. Martini, K.~J.~Swanepoel and G. Weiss, \emph{The geometry of Minkowski spaces --
a survey}. Part I. Expo.\ Math.\ \textbf{19} (2001), 97--142.
Errata: Expo.\ Math.\ \textbf{19} (2001), p.~364.

\bibitem{Robbins}
H.~Robbins, \emph{A remark on Stirling's formula}, Amer.\ Math.\ Monthly \textbf{62} (1955), 26--29.

\bibitem{SS} G.~Schechtman and A.~Shraibman,
\emph{Lower bounds for local versions of dimension reductions},
Discrete Comput.\ Geom.\ \textbf{41} (2009), 273--283. 

\bibitem{Schur}
I.~Schur, \emph{\"Uber die charakteristischen Wurzeln einer linearen Substitution mit einer Anwendung auf die Theorie der Integralgleichungen}, Math.\ Ann.\ \textbf{66} (1909), 488--510.

\bibitem{SS1}
A.~F.~Sidorenko and B.~S.~Stechkin,
\emph{Extremal geometric constants}, (Russian),
Mat.\ Zametki \textbf{29} (1981), 691--709, 798. 
English translation: Math.\ Notes \textbf{29} (1981), 352--361.

\bibitem{SS2}
A.~F.~Sidorenko and B.~S.~Stechkin,
\emph{One class of extremal geometric constants and their applications}, (Russian),
Mat.\ Zametki \textbf{45} (1989), 101--107, 128. 
English translation: Math.\ Notes \textbf{45} (1989), 253--257.

\bibitem{Sw96} K.~J.~Swanepoel,
\emph{Extremal problems in Minkowski space related to minimal networks},
Proc.\ Amer.\ Math.\ Soc.\ \textbf{124} (1996), 2513--2518.

\bibitem{Sw2}  K.~J.~Swanepoel,
\emph{Vertex degrees of Steiner Minimal Trees in $\ell_p^d$ and other smooth Minkowski spaces},
Discrete Comput.\ Geom.\ \textbf{21} (1999), 437--447.

\bibitem{Sw3}  K.~J.~Swanepoel,
\emph{Sets of unit vectors with small pairwise sums},
Quaest.\ Math.\ \textbf{23} (2000), 383--388.

\bibitem{Sw04} K.~J.~Swanepoel,
\emph{Equilateral sets in finite\dimensional normed spaces}, In: Seminar of Mathematical Analysis, eds.\ D.~Girela \'Alvarez, G.~L\'opez Acedo, R.~Villa Caro, Univ.\ Sevilla Secr.\ Publ., Seville, 2004. pp.~195--237.

\bibitem{Sw07} K.~J.~Swanepoel,
\emph{The local Steiner problem in finite\dimensional normed spaces},
Discrete Comput.\ Geom.\ \textbf{37} (2007), 419--442.

\bibitem{Sw07b} K.~J.~Swanepoel,
\emph{Upper bounds for edge-antipodal and subequilateral polytopes},
Period.\ Math.\ Hungar.\ \textbf{54} (2007), 99--106. 
\end{thebibliography}
\end{document}